\documentclass{amsproc}
\usepackage{amsfonts,amscd, amssymb}

\newtheorem{theorem}{Theorem}[section]
\newtheorem{lemma}[theorem]{Lemma}
\newtheorem{corollary}[theorem]{Corollary}
\newtheorem{proposition}[theorem]{Proposition}
\theoremstyle{remark}
\newtheorem{remark}[theorem]{Remark}
\theoremstyle{definition}
\newtheorem{definition}[theorem]{Definition}

\numberwithin{equation}{section}
\makeatother

\DeclareMathOperator{\Cdb}{{\mathbb C}}
\DeclareMathOperator{\Rdb}{{\mathbb R}}

\DeclareMathOperator{\Zdb}{{\mathbb Z}}
\DeclareMathOperator{\Ndb}{{\mathbb N}}

\begin{document}
 
\title[$C^*$-methods for operator  and Banach algebras]{Generalization of $C^*$-algebra methods via real positivity for operator  and Banach algebras} 
\author{David P. Blecher}
\address{Department of Mathematics, University of Houston, Houston, TX
77204-3008}
\email{dblecher@math.uh.edu}

\thanks{Supported by NSF grant  DMS 1201506. }

\date{\today}

\dedicatory{Dedicated with affection and gratitude to Richard V. Kadison.}

\begin{abstract}  
With Charles Read we have introduced and studied a new notion of (real) positivity in operator algebras, with an eye to extending certain $C^\ast$-algebraic results and theories to more general algebras. As motivation note that the `completely' real positive maps on $C^*$-algebras or operator systems are precisely the completely positive maps in the usual sense; however with real positivity one may develop a useful order theory for more general spaces and algebras. This is intimately connected to new relationships between an operator algebra and the $C^*$-algebra it generates.  We have continued this work together with Read, and also with Matthew Neal.  Recently with    Narutaka Ozawa
we have investigated the parts of the theory that  generalize further to  Banach algebras.  
In the present paper we describe some of this work which is connected with generalizing various 
$C^*$-algebraic techniques initiated by
Richard V. Kadison.    In particular Section 2 is in part a tribute to him in keeping with the occasion of this volume, and also discusses 
some of the origins of the theory of positivity in our sense in the setting of algebras, which the later parts of our paper developes 
further.        The most recent work will be emphasized. 
  \end{abstract}

\maketitle

\section{Introduction}    

This is a much  expanded version of our talk given at the 
AMS Special Session ``Tribute to Richard V. Kadison'' in January 2015.
We survey some of our work on a new notion of (real) positivity in operator algebras
(by which we mean closed subalgebras of $B(H)$ for 
a Hilbert space $H$), unital operator spaces, and Banach algebras, focusing on generalizing various
$C^*$-algebraic techniques initiated by
Richard V. Kadison.    In particular Section 2 is in part a tribute to Kadison in keeping with the occasion of this volume, and we will describe a small part of his opus relevant to our setting.  This section also  discusses
some of the origins of the theory of positivity  in our sense 
in the setting of algebras, which the later parts of our paper developes
further.        In the remainder of the paper
 we illustrate our real-positivity theory by showing 
how it relates to these results of Kadison, and also give some small extensions
and additional details for our recent paper with Ozawa \cite{BOZ}, and for \cite{BN15} with Neal.

With Charles Read we have introduced and studied a new notion of (real) positivity in operator algebras, with an eye to extending certain $C^*$-algebraic results and theories to more general algebras. As motivation note that the `completely' real positive maps on $C^*$-algebras or operator systems are precisely the completely positive maps in the usual sense (see
Theorem \ref{bbs} below); however with real positivity one may develop an order theory for more general spaces and algebras
that is useful at least for some purposes.     We have continued this work together with Read, and also with Matthew Neal; giving many applications.    Recently with  Narutaka Ozawa
we have investigated the parts of the theory that  generalize further to  Banach algebras.   In all of this,
 our main goal is to generalize certain nice $C^*$-algebraic results, and certain function space  or function algebra results,  which use positivity or positive approximate identities, but using our real positivity.  
As we said above, in the present paper we survey some of this work which is connected with work of Kadison.   
The most recent work will be emphasized, particularly parts of the Banach-algebraic paper \cite{BOZ}.  One reason for this emphasis is that less background is needed here
(for example noncommutative topology, or our work on noncommutative peak sets and peak
interpolation, which we have surveyed recently in 
\cite{Bnpi} although we have since made more progress  in \cite{BRord}).
Another reason is that we welcome this opportunity
to add some additional details and complements to \cite{BOZ} (and to \cite{BN15}).  
 In particular we will prove some facts that were
stated there without proof.
A subsidiary goal of Sections \ref{poroo}  and \ref{AKba} is to go through versions for general Banach algebras  of results in Sections 3, 4, and 7
of \cite{BOZ} stated
for Banach algebras with approximate identities.            
We will also pose several open questions.
The drawback of course with this focus 
is that the Banach algebra case is sometimes less impressive and clean
 than the operator 
algebra case, there usually being a price to be paid for generality.    
  
Of course an operator algebra or function algebra $A$ may have no positive elements in
the usual sense.   However we  see e.g.\ in Theorem \ref{brcof2} below
that an operator algebra
 $A$ has a contractive approximate identity iff there is a great abundance of
 real-positive elements;
for example, iff $A$ is spanned by its real-positive elements.   Below Theorem \ref{brcof2} we will point out that
this is also true for certain classes of Banach algebras.
Of course in the theory of $C^*$-algebras, positivity and the existence of positive approximate identities  are crucial.
Some form of our `positive cone'  already appeared in papers of Kadison and Kelley and Vaught in the early 
1950's, and in retrospect it is a natural idea to attempt to use such a cone to 
generalize various parts of  $C^*$-algebra theory involving positivity and the existence of positive approximate identities.  However nobody seems to have pursued this until now.      
 In 
 practice, some things are much harder than the $C^*$-algebra case.    And many things simply do not generalize beyond the $C^*$-theory; that is, our approach is effective at generalizing some parts of $C^*$-algebra theory, but
not others.     The worst problem is that although we have a functional calculus, it is not as good.  
Indeed often at first sight in a given $C^*$-subtheory, nothing seems to work.  But if one looks a little closer something works, or an interesting conjecture is raised.    Successful applications so far include for example
 noncommutative topology (eg.\ noncommutative Urysohn and Tietze theorems
for general operator algebras, and the theory of open, closed and compact projections
in the bidual), lifting problems,  the structure of completely contractive idempotent maps on 
an operator algebra (described in Section 3 below),  noncommutative peak sets,
 peak interpolation, and some other noncommutative function theory, comparison theory,  
 the structure of operator algebras,  new relationships between an operator algebra and the $C^*$-algebra it generates, approximate identities, etc.

\section{Richard Kadison and the beginnings of positivity}   
\label{kad}

The first published words of Richard V. Kadison appear to be the following:  \begin{quotation} ``It is the purpose of the present note to investigate
the order properties of self-adjoint operators individually and with
respect to containing operator algebras''.\end{quotation}    This was from the paper  \cite{KadOrderp}, which appeared in 1950.    In the early 1950s the war was over, John von Neumann was editor of the Annals of Mathematics and was talking to anybody who was interested about 
`rings of operators',  Kadison was in Chicago and the IAS, and all was well with the world.     In 1950, von Neumann  wrote a letter to Kaplansky (IAS Archives, reproduced in \cite{vNletters}) which begins as follows:

\medskip

\begin{quotation} 
``Dear Dr. Kaplansky,

\smallskip

Very many thanks for your letter of February 11th and your manuscript on "Projections in 
Banach Algebras". I am very glad that you are submitting it for THE ANNALS, and I will 
immediately recommend it for publication.

Your results are very interesting.  You are, of course, very right: 
 I am and I have been 
for a long time strongly interested in a ``purely algebraical" rather than 
``vectorial-spatial" foundation for theories of operator-algebras or operatorlike-algebras.  To be more precise: It always seemed to me that there were three successive 
levels of abstraction - first, and lowest, the vectorial-spatial, in which the Hilbert 
space and its elements are actually used; second, the purely algebraical, where only the 
operators or their abstract equivalents are used; third, the highest, the approach when 
only linear spaces or their abstract equivalents (i.e.\ operatorially speaking, the 
projections) are used.  [...] After Murray and I had reached somewhat rounded results on 
the first level, I neglected to make a real effort on the second one, because I was 
tempted to try immediately the third one. This led to the theory of continuous 
geometries. In studying this, the third level, I realized that one is led there to the 
theory of ``finite" dimensions only.   The discrepancy between what might be considered 
the ``natural" ranges for the first and the third level led me to doubt whether I could 
guess the correct degree of generality for the second one...".
\end{quotation}

\medskip

It is remarkable here to recall that von Neumann invented the abstract definition  of Hilbert  spaces, the theory of unbounded operators (as well as much of the bounded theory), 
ergodic theory, the mathematical formulation of 
quantum mechanics, many fundamental concepts associated with groups  (like amenability), and several other fields of analysis.  
Even today, teaching a course in functional analysis 
can sometimes feel like one is mainly expositing the work of this one man.    However
von Neumann is saying above  that he had unfortunately neglected what he calls the `second level' of `operator algebra', and at the time of this letter this was ripe and timely for exploration.   

Happily, about the time the above letter was written, Richard Vincent Kadison entered the world with a bang: a spate of amazing papers at von Neumann's `second level'.    Indeed  
Kadison soon  took leadership of the American school
of operator algebras.  
Some part of his early work was concerned with positive cones and their properties.
We will now briefly describe  a few of these and spend much of the remainder of our article showing how they can be generalized
to nonselfadjoint operator algebras and Banach algebras.   The following comprises just  a tiny part
of Kadison's opus; but nonetheless is still foundational and seminal.  Indeed much of $C^*$-algebra theory 
would disappear without this work.   At the start of
this section we already mentioned his first paper, devoted to `order properties of self-adjoint operators individually and with
respect to containing operator algebras'.   His memoir
``A representation theory for commutative topological algebra" \cite{KadMem} soon
followed, one small aspect of which was the introduction and study of positive cones, states, and square roots in Banach algebras.   In the 1951 Annals paper  \cite{Kiso}, Kadison generalized the ‘Banach-Stone’
theorem, characterizing surjective isometries between $C^*$-algebras. This result
has inspired very many functional analysts and innumerable papers.  See for example \cite{FJ} for a collection
of such results, together with their history, although this reference is a bit dated since the list grows all the time.  
See also e.g.\ \cite[Section 6]{Bmca}.  
In a 1952 Annals paper \cite{KadSch}
 he proved the Kadison--Schwarz inequality, a fundamental
inequality satisfied by positive linear maps on $C^*$-algebras.      His student St{\o}rmer continued this in a very long (and still continuing) 
series of deep papers.  Later this Kadison--Schwarz work was connected to {\em completely positive maps}, Stinespring's theorem and Arveson's extension theorem (see the next paragraph and e.g.\ \cite{Pbook}), conditional expectations, operator systems and operator spaces, quantum information theory, etc.    A related enduring interest 
of Kadison's is  projections and conditional expectations on $C^*$-algebras and von Neumann algebras.
A search of his collected works finds very many contributions to this topic (e.g.\ \cite{Kadce}).

In 1960,
 Kadison together with I. M. Singer \cite{KS}  initiated the study of 
{\em nonselfadjoint operator algebras} on a Hilbert space (henceforth simply called
{\em operator algebras}).  Five years later or so, the late Bill Arveson in his thesis continued  the
study of nonselfadjoint operator algebras, using heavily the
Kadison-Fuglede determinant of \cite{KadFug} and positivity 
properties of conditional expectations.    This work was published in \cite{AIOA}; it developes a 
von Neumann algebraic theory of noncommutative  Hardy spaces.
We mention in passing that we continued Arveson's work from  \cite{AIOA}
in a series of papers with Labuschagne, again using the
Kadison-Fuglede determinant of \cite{KadFug} as a main tool (see e.g.\ the survey \cite{BL5}), as well as
positive conditional expectations and the Kadison--Schwarz inequality.  
This is another example of using $C^*$-algebraic methods,
 and in particular  tools originating in seminal work of
 Kadison, in a more general (noncommutative function theoretic) setting.
However since this lies in a different direction to the rest of the present article we will say no more about this.
In the decade after  \cite{AIOA}, Arveson went on to write many other 
seminal papers on nonselfadjoint operator algebras, perhaps most notably 
 \cite{SOC}, in which completely positive maps and the Kadison--Schwarz inequality
play a decisive role, and which may be considered a source of the later theory of
{\em operator spaces} and {\em operator systems}.

Another example: in 1968 Kadison and Aarnes, his first student at Penn, introduced {\em strictly positive} elements in a $C^*$-algebra $A$, namely $x \in A$ which satisfy
 $f(x) > 0$ for every state $f$
of $A$.     They proved the fundamental basic result:

\begin{theorem}[Aarnes--Kadison] \label{AaKa}  For a
$C^*$-algebra $A$ the following are equivalent:
\begin{itemize} \item [(1)]   $A$ has a strictly positive element.
\item [(2)]   $A$ has a countable increasing contractive approximate identity.
 \item [(3)]   $A = \overline{zAz}$ for some positive $z \in A$.
\item [(4)]   The positive cone $A_+$ has an element $z$ of full support  (that is, the support projection $s(z)$ is $1$).
\end{itemize}
The approximate identity in {\rm (2)} may be taken to be commuting, indeed it may be taken to be
$(z^{\frac{1}{n}})$ for $z$ as in {\rm (3)}.   If $A$ is a separable $C^*$-algebra then these all hold.
\end{theorem}

Aarnes and Kadison did not prove (4).  However (4) is immediate from the rest since 
$s(z)$ is the weak* limit of $z^{\frac{1}{n}}$, and the converse is
easy.  
This result is related to the theory of hereditary subalgebras, comparison theory in 
$C^*$-algebras, etc.   In fact much of modern $C^*$-algebra
theory would collapse without basic results like this.    For example, the Aarnes--Kadison theorem  implies the beautiful characterization due to 
Prosser \cite{Pros} of closed one-sided ideals in a separable $C^*$-algebra $A$ as the `topologically principal  (one-sided) ideals' (we are indebted to the referee for pointing out that Prosser was a student of Kelley).   The latter is 
equivalent to the characterization of hereditary subalgebras of such  $A$ as the 
subalgebras of form $\overline{zAz}$.    (We recall that a hereditary subalgebra, or HSA for short,
is a closed selfadjoint subalgebra $D$ satisfying $DAD \subset D$.)  These results are used in many modern theories
such as that of 
the Cuntz semigroup.    Or, as another example, the Aarnes--Kadison theorem  is used in the important stable 
isomorphism theorem for Morita equivalence of $C^*$-algebras (see e.g.\ \cite{Bla,Brown}).

 Indeed in some sense  the Aarnes--Kadison theorem  is 
equivalent to the first assertion of the following:

\begin{theorem} \label{AaKa2}    A HSA  (resp.\ closed right ideal) in a
$C^*$-algebra $A$ is (topologically) principal, that is of the form   $\overline{zAz}$   (resp.\ $\overline{zA}$)
for some $z \in A$ iff it has a countable (resp.\ countable left) contractive  
approximate identity.   Every closed right ideal   (resp.\ HSA) is the closure of an increasing union 
of such (topologically) principal right ideals   (resp.\ HSA's).     
\end{theorem}

Indeed separable HSA's  (resp.\ closed right ideals) in $C^*$-algebras have countable 
(resp.\ countable left) approximate identities.

One final work of Kadison which we will mention here is his first
paper with Gert Pedersen \cite{KP}, which amongst other things initiates the development of a  comparison theory for elements in $C^*$-algebras  generalizing the von Neumann  equivalence  of projections.   Again positivity and properties of the positive cone are key to that work.      This paper is often cited 
in recent papers on the Cuntz semigroup.    

The big question we wish to address in this article is how to generalize such 
results and theories, in which positivity is
the common theme, to not necessarily selfadjoint
operator algebras (or perhaps  even  Banach algebras).    In fact one often can, as we have shown in joint work 
with Charles Read, Matt Neal, Narutaka Ozawa, and others.  
  In the Banach algebra  literature of course there are many generalizations of $C^*$-algebra results, but as far as we are aware there is no `positivity' approach like ours (although there is a
trace of it in \cite{Est}).   In particular we mention 
Sinclair's generalization from \cite{Sinc}
of part of the Aarnes--Kadison theorem:

\begin{theorem}[Sinclair]   \label{AaKaSin}   A separable 
Banach algebra $A$ with a bounded approximate identity has a commuting bounded approximate identity.
\end{theorem}

If $A$ has a  countable bounded approximate identity then Sinclair and others show results like
  $A = \overline{xA} = \overline{Ay}$ for some $x, y \in A$.  In part of our work we follow Sinclair in using variants of the proof of the Cohen factorization method to achieve such results but with `positivity'.

We now explain one of the main ideas.
Returning  to the early 1950s: it was only then becoming  perfectly clear what a $C^*$-algebra was; a few fundamental facts about the positive cone were still being proved.  We 
recall that an unpublished result of  
Kaplansky removed the final superfluous abstract axiom for
a $C^*$-algebra, and this used a result in a 1952
paper of Fukamiya, and in a 1953 paper of John Kelley and Vaught \cite{KV}
based  on a 1950 ICM talk by those authors.    These sources are 
referenced in almost  every $C^*$-algebra book.
The paper of Kelley and Vaught was titled
 ``The positive cone in Banach algebras'', and in the first section of the 
paper they
discuss precisely that.  The following
 is not an important part of their paper, but as in Kadison's paper a year earlier they have a small discussion on how to
 make sense of the notion of a positive cone in a Banach algebra, and they 
prove some basic results here.  
 Both Kadison and Kelley and Vaught have some use for the set
$${\mathfrak F}_A = \{ x \in A : \Vert 1 - x \Vert \leq 1 \}.$$
In their case $A$ is unital (that is has an identity of norm $1$), but if not 
one may take $1$ to be the identity of a unitization of $A$.
In \cite{BRI}, Charles Read and the author began a study of {\em not necessarily selfadjoint operator algebras on a Hilbert space $H$}; henceforth {\em operator algebras}.
In this work,
 ${\mathfrak F}_A$ above plays a pivotal role, and also the cone $\Rdb^+ {\mathfrak F}_A$.
In \cite{BRII} we looked at the slightly larger 
cone ${\mathfrak r}_A$ of so called {\em accretive elements} (this is a non-proper cone or `wedge').    
In an operator algebra these are the elements with
positive real part;  in a general Banach algebra they are the elements $x$ 
with Re$ \; \varphi(x) \geq 0$ for every state $\varphi$ on a unitization of $A$.   We recall that
a state on a unital Banach algebra
is, as usual in the theory of numerical range \cite{BoNR1},  a 
norm one functional $\varphi$ such that $\varphi(1)=1$.    That is, accretive elements
 are the elements with  numerical range in the closed right half-plane.   We sometimes
also call these the {\em real positive} elements.      We will see later in Proposition \ref{whau} that 
$\overline{\Rdb^+ {\mathfrak F}_A} = {\mathfrak r}_A$.     That is, the one cone above is the closure of the other.
 We write ${\mathfrak C}_A$ for either of these cones.

The following is known, some of it attributable to Lumer and Phillips,
or implicit in  the theory of contraction semigroups, or can be found in e.g.\ \cite[Lemma 2.1]{Mag}.
The latter paper was no doubt  influential on our real-positive theory in \cite{BOZ}.  

\begin{lemma} \label{chaccr}   Let $A$ be a unital Banach algebra.   If $x \in A$ the following are equivalent:
 \begin{itemize} \item [(1)]    $x \in {\mathfrak r}_A$, that is, $x$ has  numerical range in the closed right half-plane.
 \item [(2)]  
   $\Vert 1 - tx \Vert \leq 1 + t^2 \Vert x \Vert^2$  
for all $t > 0$.  
 \item [(3)]    $\Vert \exp(-tx) \Vert \leq 1$ for all $t > 0$.
\item [(4)] $\Vert (t + x)^{-1} \Vert \leq \frac{1}{t}$  
for all $t > 0$.  
\item [(5)]    $\Vert 1 - tx \Vert \leq \Vert 1 - t^2 x^2 \Vert$  for all $t > 0$.
\end{itemize}  
    \end{lemma}

\begin{proof}     
 For the equivalence of (1) and (3), see \cite[p.\ 17]{BoNR1}.
Clearly (5) implies (2).
That  (2) implies (1) follows by applying a state $\varphi$ to see
$|1 - t \varphi(x)| \leq 1 + K  t^2$, which forces Re$\,  \varphi (x) \geq 0$) (see  \cite[Lemma 2.1]{Mag}).    
Given (4) with $t$ replaced by $\frac{1}{t}$, we have
$$\Vert 1 - tx \Vert = \Vert (1 + tx)^{-1} (1 + tx)(1 - tx) \Vert \leq \Vert 1 - t^2 x^2 \Vert.$$
This gives (5).      Finally (1) implies (4) by e.g.\ Stampfli 
and Williams  result \cite[Lemma 1]{SW} that the norm in (4) is
dominated by the reciprocal of the distance  from $-t$ 
to the numerical range of $x$. 
\end{proof}  

(We mention another equivalent condition: given $\epsilon > 0$ there exists  a $t > 0$ with $\Vert 1 - tx \Vert < 1 + \epsilon t$.
See e.g.\ \cite[p.\ 30]{BoNR1}.)

 Real positive elements, and the smaller  set ${\mathfrak F}_A$ above,  will play the role for us of positive elements in a $C^*$-algebra.  
While they are not the same,
real positivity is very compatible with the 
usual definition of positivity in a $C^*$-algebra, as  will be seen very clearly in the sequel, and in particular in the next section.

\section{Real completely positive maps and projections}

Recall that a linear map $T : A \to B$ between $C^*$-algebras (or operator systems) is {\em completely
positive} if $T(A_+) \subset B_+$, and similarly at the matrix levels.
By a {\em unital operator space} below we mean a subspace
of $B(H)$ or a unital $C^*$-algebra containing the identity.   We gave 
abstract characterizations of these objects with
Matthew Neal in \cite{BNc1, BNc2}, and have studied them elsewhere.

\begin{definition}   A linear map $T : A \to B$ between operator
algebras or unital operator spaces is {\em real positive} if $T({\mathfrak r}_{A}) \subset {\mathfrak r}_{B}$.  It is {\em real completely
positive}, or {\em  RCP} for short, if $T_n$ is
real positive on $M_n(A)$ for all $n \in \Ndb$.
  \end{definition}

(This and the following two results are later variants from \cite{BBS}
of matching material from \cite{BRI} for ${\mathfrak F}_{A}$.)

\begin{theorem}   \label{bbs}  A (not necessarily unital) linear map $T : A \to B$ between $C^*$-algebras or operator
systems is completely positive in the usual sense iff it is RCP.
\end{theorem}

We say that 
an algebra is {\em approximately unital}
 if it has a contractive approximate
identity (cai).

\begin{theorem}[Extension and Stinespring-type Theorem]   \label{bbs2}  A linear map $T : A \to B(H)$ on
an  approximately unital operator
algebra or unital operator space is   RCP iff $T$ has a completely positive (in the usual sense)
extension $\tilde{T} : C^*(A) \to B(H)$.   Here $C^*(A)$ is a $C^*$-algebra generated by $A$.
This is equivalent to being able
to write $T$ as the restriction to $A$ of $V^* \pi(\cdot) V$ for a
$*$-representation $\pi : C^*(A) \to B(K)$, and an operator $V : H \to
K$.   \end{theorem}

Of course this result  is closely related to Kadison's Schwarz inequality.
In particular, if one is trying to generalize results where
completely positive maps and the Kadison's Schwarz inequality are used in the $C^*$-theory,
to operator algebras, one can see how Theorem \ref{bbs} would play a key role.    And indeed it does,
for example in the remaining results in this section.  

 We will not say more about unital operator spaces in the present article, except to say that it
is easy to see that completely contractive  unital maps on a   unital operator space are RCP.

We give two or three applications from  \cite{BN15} of Theorem \ref{bbs2}.  The first is related to Kadison's Banach--Stone theorem for $C^*$-algebras \cite{Kiso}, and uses our Banach--Stone type theorem 
\cite[Theorem 4.5.13]{BLM}. 

\begin{theorem}  \label{rpBS} {\rm (Banach--Stone type theorem) } \  Suppose that $T : A \to B$ is a  completely isometric surjection between approximately unital operator algebras.  Then $T$ is  real (completely)  
positive if and only if $T$ is an algebra homomorphism.
\end{theorem}   
 
In the following discussion, by a  projection $P$ on an operator algebra $A$, we mean an
idempotent linear map $P : A \to A$.  We say that $P$ is a {\em conditional expectation} if  
$P(P(a)bP(c)) = P(a) P(b) P(c)$ for $a, b,  c \in A$.

\begin{proposition}  \label{rcpun}   A  real
completely positive completely contractive  map (resp.\ projection)
on an approximately unital operator algebra $A$, extends to a unital completely contractive  map (resp.\ projection)
on the unitization $A^1$.  \end{proposition}  

 Much earlier, we studied completely contractive  projections $P$ and conditional expectations
on unital operator algebras.    Assuming that  $P$ is also unital (that is, $P(1) = 1$)
and that Ran$(P)$ is a subalgebra, we showed (see e.g.\ \cite[Corollary 4.2.9]{BLM}) that $P$ is a conditional expectation.   This is the operator algebra
variant of  Tomiyama's  theorem  for $C^*$-algebras.   A well known result of 
Choi and Effros states that 
the range of a completely positive 
projection $P : B \to B$ on a $C^*$-algebra $B$, is again a $C^*$-algebra with
product $P(xy)$.   The analogous result for unital completely contractive 
projections on unital operator algebras is true too, and is implicit in 
the proof of our generalization of Tomiyama's theorem above.   Unfortunately, there is no analogous result 
for (nonunital) completely contractive 
projections on possibly nonunital operator algebras
without adding extra hypotheses on $P$.     However if we add the condition that
$P$ is also `real completely positive', then the question does make good sense and one can easily 
deduce from the unital case and Proposition \ref{rcpun} one direction of the following:

\begin{theorem}  \label{rcpun2} {\rm \cite{BN15}} \    The range of a  completely contractive projection 
$P : A \to A$ on an approximately unital operator algebra  is again an operator algebra with
product $P(xy)$ and cai $(P(e_t))$ for some cai $(e_t)$ of $A$, 
 iff $P$ is real completely positive.    
 \end{theorem}

\begin{proof}    For the `forward direction' note that $P^{**}$ is a unital complete contraction,
and hence is real completely positive as we said in above Theorem \ref{rpBS}.
 For the `backward direction'  the following proof, due to the author 
and Neal, was originally a remark in \cite{BN15}.  By passing to the bidual we may assume that
$A$ is unital.
using the BRS characterization of operator algebras \cite[Section 2.3]{BLM}. 
If $P(P(1) x) = P(x P(1))  = x$ for all $x \in {\rm Ran}(P)$ then
we are done
by the abstract characterization of operator algebras from \cite[Section 2.3]{BLM}, since 
then $P(xy)$ defines a bilinear completely contractive 
product on ${\rm Ran}(P)$ with `unit' $P(1)$.
Let $I(A)$ be the injective envelope of $A$.  We may extend $P$ to a completely positive 
completely contractive
map $\hat{P} : I(A) \to I(A)$, by \cite[Theorem 2.6]{BBS} and injectivity of $I(A)$.  
We will abusively sometimes write  
$P$ for $\hat{P}$, and also for its second adjoint
on $I(A)^{**}$.  The latter is also completely  positive and completely contractive. Then  
$$P(P(1)^{\frac{1}{n}}) \geq P(P(1)) = P(1) \geq P(P(1)^{\frac{1}{n}}).$$
Hence these quantities are equal.
In the limit, $P(s(P(1))) = P(1)$, if $s(P(1))$ is the support projection
of $P(1)$.  Hence $P(z) = 0$
where $z = s(P(1)) - P(1)$.  
If $y \in I(A)_+$ with $\Vert  y \Vert \leq 1$, then
$P(y) \leq P(1) \leq s(P(1))$, and so 
$s(P(1)) P(y) = P(y) = P(y) s(P(1))$.
It follows that $s(P(1))x = x s(P(1)) = x$ for   
all $x \in {\rm Ran}(\hat{P})$.
If also $\Vert x \Vert \leq 1$, then
$$P(P(1) x) = P(s(P(1))x) - P(z x) = P(s(P(1))x) = P(x)$$   
by the Kadison-Schwarz inequality, since 
$$P(z x) P(z x)^* \leq P(z x x^* z) \leq P(z^2) \leq P(z) = 0 .$$
Thus $P(P(1) x) = x$ if $x \in P(A)$. Similarly, $P(x P(1))  = x$  as desired. 
Thus $P(xy)$ defines a bilinear completely contractive
product on ${\rm Ran}(P)$ with `unit' $P(1)$.
\end{proof}

The main thrust of  \cite{BN15} is the investigation of the completely contractive projections and conditional 
expectations,
and in particular the `symmetric  projection problem' and the `bicontractive
  projection problem', in the category of operator algebras, attempting 
to  find operator algebra generalizations of 
certain deep results of St{\o}rmer, Friedman and Russo, 
Effros and St{\o}rmer, Robertson and Youngson, 
and others (see papers of these authors referenced in the bibliography below), concerning projections and their ranges, assuming in addition that our projections are  real completely positive.   We say that an idempotent linear $P : X \to X$ is  
{\em completely symmetric} (resp.\ {\em completely bicontractive})  if $I - 2P$ is completely contractive
(resp.\  if $P$ and $I - P$ are completely contractive).   `Completely  symmetric' implies  `completely bicontractive'.
The two problems mentioned at the start of this paragraph concern  1)\  Characterizing such   projections  $P$;
or 2)\  characterizing the range of such projections.     On a unital $C^*$-algebra $B$ the work of some of the authors mentioned at the start of this paragraph establish that 
  unital positive bicontractive projections are also symmetric, and are 
precisely $\frac{1}{2}(I + \theta)$, for a period $2$ $*$-automorphism
$\theta : B \to B$.   
The possibly nonunital positive bicontractive projections $P$ are of a similar form, and 
then $q = P(1)$ is a central projection in $M(B)$ with respect to which $P$ decomposes into a direct sum of $0$ and
a projection of the above form $\frac{1}{2}(I + \theta)$, for a period $2$ $*$-automorphism
$\theta$ of $qB$.   Conversely, a map $P$ of the latter form is automatically  
completely bicontractive, and the range of $P$, which is
  the set of fixed points of $\theta$, is a $C^*$-subalgebra,
and $P$ is a conditional expectation.

One may ask what from the last paragraph is
true for general (approximately unital) operator algebras $A$?  
The first thing to note is that 
now `completely bicontractive' is   no longer  the same as `completely symmetric'.   The following 
is our solution to  the  {\em  symmetric projection problem}, and it uses 
Kadison's Banach--Stone theorem for $C^*$-algebras  \cite{Kiso}, and our variant of the latter
for approximately unital operator algebras (see e.g.\ \cite[Theorem 4.5.13]{BLM}):

\begin{theorem}  \label{rcpst} {\rm \cite{BN15}} \ Let $A$ be an approximately unital operator algebra,
and $P : A \to A$ a completely symmetric real completely positive projection. 
Then the range of $P$  is an approximately unital 
subalgebra of $A$.   Moreover, $P^{**}(1) = q$ is a projection in the multiplier algebra
$M(A)$ (so is both open and closed).  

Set $D = qAq,$ a hereditary subalgebra of $A$ containing $P(A)$.
There exists a period $2$ surjective completely 
isometric homomorphism $\theta : A \to A$ such that $\theta(q) = q$, so that $\theta$ restricts to 
a  period $2$ surjective completely 
isometric homomorphism $D \to D$.    Also,  $P$ is
the  zero map on $q^\perp A + A q^\perp + q^\perp A q^\perp$, and $$P =  \frac{1}{2} (I + \theta) \; \; \;
{\rm on} \; D.$$
In fact $$P(a)=  \frac{1}{2} (a +  \theta(a) (2q-1)) \, , \qquad a \in A.$$    The range of $P$ is the set of fixed points of $\theta$.  

Conversely, any 
map of the form in the last equation is a 
completely symmetric real completely positive projection.  \end{theorem}

{\bf Remark.}  In the case that $A$ is unital but $q$ is not central in the last theorem, if one solves the last equation for $\theta$, and then examines what it means that $\theta$ is 
a homomorphism, one obtains some interesting algebraic formulae involving $q, q^\perp, A$ and $\theta_{\vert qAq}$. 

\bigskip

For the more general class of 
completely bicontractive projections, 
a first look is disappointing--most of the last paragraph no longer works in general.   One does not 
always get an associated completely isometric automorphism $\theta$  such that $P = \frac{1}{2}(I + \theta)$, and
$q = P(1)$ need not be a central projection.  
However, as also
seems to be sometimes the case when attempting to generalize a given $C^*$-algebra fact to 
more general algebras, a closer look at the result, and at examples, does uncover an interesting question.  
Namely,  given  an approximately unital operator algebra $A$ and a real completely positive projection $P : A \to A$
which is completely bicontractive, when is the range of $P$ a 
 subalgebra of $A$ and $P$ a conditional expectation?    This seems to be  the right version of the
 `bicontractive projection problem' in the 
operator algebra category.     
We give in \cite{BN15} a sequence of three 
reductions that reduce the question.   The first reduction is that by passing to the bidual we may assume
that the algebra $A$ is unital.   The second reduction is that by cutting down to $qAq$, where $q = P(1)$ (which
one can show is a projection), we may
further assume that  $P(1) = 1$ (one can show $P$ is zero on $q^\perp A + Aq^\perp$).   
The third  reduction is  by restricting attention to the closed algebra generated by $P$, we may further assume that 
$P(A)$ generates $A$ as an operator algebra.   We call this the `standard position' for 
the bicontractive projection problem.   It turns out that when in standard position, Ker$(P)$ is forced to be
an ideal with square zero.      

In the second reduction above, that is if $A$ and $P$ are unital, then
one may show that $A$ decomposes as $A = C \oplus B$, where $1_A \in 
B = P(A), C = (I-P)(A)$,
and we have the relations $C^2 \subset B, C B + B C \subset C$ (see \cite[Lemma 4.1]{BN15} and its proof).
The period 2 map $\theta : x + y \mapsto x-y$ for $x \in B, y \in C$ is a homomorphism (indeed an automorphism)
on $A$ iff $P(A)$ is a subalgebra of $A$, and we have, similarly to  
Theorem \ref{rcpst}: 

\begin{corollary} \label{bipe}  If $P : A \to A$ is a unital idempotent on a unital operator algebra then $P$ is
completely bicontractive iff there is a  period 2 linear surjection $\theta : A \to A$ such that $\Vert I \pm \theta \Vert_{\rm cb} \leq 2$ and 
$P = \frac{1}{2}(I + \theta)$.   The range of $P$ is a subalgebra iff $\theta$ is also a homomorphism, and then
the range of $P$ is the set of fixed points of this automorphism $\theta$.  
Also, $P$ is completely symmetric iff $\theta$ is completely contractive.  \end{corollary}  

We remark that for the subcategory of uniform
algebras (that is, closed unital (or approximately unital)
subalgebras of $C(K)$, for compact $K$), there is a complete solution to the
bicontractive  projection problem.

\begin{theorem}  \label{rcpua}  Let $P : A \to A$ be a  real  positive 
 bicontractive projection on a (unital or approximately 
unital) uniform algebra.  Then  $P$ is symmetric,
and so of course by Theorem {\rm   \ref{rcpst}} we have
that $P(A)$ is a subalgebra of $A$,
and $P$ is a conditional expectation.    \end{theorem}

\begin{proof}    We sketch the idea, found in a  conversation with Joel Feinstein.
By the first two reductions described above we can assume that 
$A$ and $P$ are unital.   We also know that $B = P(A)$ is a subalgebra, since if it were not then the third
reduction described above would yield nonzero nilpotents, which cannot exist in a function algebra. 
Thus by the discussion above the theorem, the map $\theta(x+ y) = x-y$ there is an algebra automorphism of $A$,
hence an isometric isomorphism (since norm equals spectral radius).  So  $P = \frac{1}{2} (I + \theta)$
is symmetric.   
\end{proof}

The same three step reduction shows that we can also solve the problem in the affirmative for real completely 
positive completely bicontractive projections $P$ on a unital operator algebra $A$ 
such that the closed algebra generated by $A$ is semiprime (that is, it has no nontrivial
square-zero ideals).   We have found counterexamples to the general question, but we
have also have found conditions that make all known (at this point) counterexamples go away.   See
\cite{BN15} for details.

\section{More notation, and existence of `positive' approximate identities} 
\label{mnepai}

We have already defined the cone ${\mathfrak r}_A$ of accretive or 
`real positive' elements, and its dense subcone $\Rdb^+ {\mathfrak F}_A$.
  Another subcone which is occasionally of interest is the cone consisting of 
elements of $A$ which are `sectorial' of angle $\theta< \frac{\pi}{2}$.    
For the purposes of this paper
being sectorial of angle $\theta$ will mean  that the numerical range in  $A$  (or in a unitization of $A$ if 
$A$ is nonunital) 
is contained in the sector $S_\theta$ consisting of complex numbers $r e^{i \rho}$ with $r \geq 0$ and $|\rho| \leq \theta$.  
This third cone is a dense subset of the second cone $\Rdb^+ {\mathfrak F}_A$
if $A$ is an operator algebra \cite[Lemma 2.15]{BRord}.
We remark that there exists a well established functional  calculus for sectorial operators (see e.g.\ \cite{Haase}).
Indeed the advantages of this cone and the last one seems to be mainly that these have better functional 
calculi.  For the cone $\Rdb^+ {\mathfrak F}_A$, if $A$ is an operator algebra, one could use the functional 
calculus  coming from von Neumann's inequality.  Indeed if 
$\Vert I - x \Vert \leq 1$ then $f \mapsto f(I-x)$ is a contractive homomorphism on the disk algebra.
If $x$ is real positive in an operator algebra, one could also use Crouzeix's remarkable functional calculus on the 
numerical range of $x$ (see e.g.\ \cite{Crou08}).   If $x$ is sectorial in a Banach algebra, one may use 
the functional calculus 
for sectorial operators \cite{Haase}.

A final notion of positivity which we introduced in 
the work with Read, which is slightly more esoteric, but which is a close approximation to
 the usual $C^*$-algebraic notion of positivity:   In the theorems below we will sometimes say that an element $x$ is {\em nearly positive}; 
   this means that in the statement of that result, given $\epsilon >0$ one can also choose the element in that statement to be  real positive  and within $\epsilon$ of its real part (which is positive in the usual sense).     In fact whenever we say 
`$x$ is  nearly positive'
below, we are in fact able, for any  given $\epsilon >0$,  to choose $x$ to also be a contraction with numerical range 
within a thin `cigar' centered on the line segment $[0,1]$ of height $< \epsilon$.   That is,
$x$ has sectorial angle $< \arcsin \epsilon$.     In an operator algebra any contraction $x$ with 
such a sectorial angle is accretive and satisfies $\Vert x - {\rm Re} \, x \Vert \leq \epsilon$, so $x$ is 
within $\epsilon$ of an operator which is positive in the usual sense. 
  Indeed if $a$ is an accretive element in an operator algebra then (principal) $n$-th roots of $a$ have spectrum and numerical radius within a
 sector $S_{\frac{\pi}{2n}}$, and hence are as close as we like (for $n$ sufficiently large) to an operator which is positive in the usual sense (see Section \ref{poroo}).  
Thus one obtains `nearly positive elements' by taking $n$-th roots of  accretive elements.    A {\em nearly positive approximate identity}     $(e_t)$ means that it is real positive
and  the sectorial angle of $e_t$ converges to $0$ with $t$.   
(We remark that at the time of writing we do not know
%QQ
 for general Banach algebras
if roots (or $r$th powers for
$0 < r < 1$) of 
accretive elements are in $\Rdb^+ {\mathfrak F}_A$ or in
the third cone in the last paragraph, or if that
 third cone is contained in the second cone.)   

In the last paragraphs we have described several variants of `positivity', which at least in an operator algebra are each  
successively stronger than the last.    It is convenient to mentally picture each of these notions by sketching the region containing the
 numerical range of $x$.   Thus for the first notion, the  accretive elements, one simply pictures the right 
half plane in $\Cdb$.   One pictures the second, the cone $\Rdb^+ {\mathfrak F}_A$, as a dense cone 
in the right half plane composed of closed disks center $a$ and radius $a$, for all $a > 0$.   The third cone 
is pictured as increasing sectors $S_\theta$ in $\Cdb$, for increasing $\theta < \frac{\pi}{2}$.   And the 
`nearly positive' elements are pictured by the thin `cigar' mentioned a  paragraph or so back,
centered on the line segment $[0,1]$ of height $< \epsilon$, and contained
in the closed disk center $\frac{1}{2}$ of radius $\frac{1}{2}$.

We now list some more of our notation and general facts: We write ${\rm Ball}(X)$ for the set $\{ x \in X : \Vert x \Vert \leq 1 \}$.     
For us Banach algebras satisfy $\Vert xy \Vert \leq \Vert x \Vert \Vert y \Vert$.   If $x \in A$
for a Banach algebra $A$, then ba$(x)$ denotes the closed subalgebra generated
by $x$.  
If $A$ is a Banach algebra which is not Arens regular, then the multiplication we usually use
on $A^{**}$ is the  `second Arens product' ($\diamond$ in the notation of \cite{Dal}).
This is weak* continuous in the second variable.
 If $A$ is a nonunital, not necessarily Arens regular,  Banach algebra with a bounded approximate identity
(bai),
then $A^{**}$ has a so-called `mixed identity'  \cite{Dal,Pal,DW}, which we will again write as $e$.  This is a right  identity for
the first Arens product, and a left identity for
the second Arens product.   A mixed identity need not
be unique, indeed  mixed identities are just the weak* limit points of bai's for $A$.

See the book of Doran and Wichmann \cite{DW} for a compendium of
 results about approximate identities and related topics.    
If $A$ is an approximately unital  Banach algebra, then 
the left regular representation embeds $A$ isometrically in $B(A)$.  
We will always write $A^1$ for the {\em multiplier unitization} of $A$,
that is, we identify $A^1$ isometrically with $A + \Cdb I$ in
$B(A)$.      Below $1$ will almost always denote the identity of $A^1$, if $A$ is not already unital.  
If $A$ is a nonunital, approximately unital  Banach algebra
then the multiplier  unitization $A^1$ may also be identified isometrically
with the subalgebra  $A + \Cdb e$ of $A^{**}$ for a fixed 
 mixed identity $e$  of norm $1$ for $A^{**}$.

 We recall  that a subspace $E$ of a Banach space $X$ is an $M$-ideal in
$X$ if $E^{\perp \perp}$ is complemented in $X^{**}$ via a contractive projection $P$ so that 
$X^{**} = E^{\perp \perp} \oplus^{\infty} {\rm Ker}(P)$.   In this case there is a unique contractive projection  onto $E^{\perp \perp}$.   This concept 
was invented by Alfsen and Effros,
and  \cite{HWW} is the basic text for their beautiful and powerful theory.    
 By an {\em $M$-approximately unital Banach algebra}
we mean a Banach algebra which is an $M$-ideal in
its multiplier unitization $A^1$.  This is equivalent (see \cite[Lemma 2.4]{BOZ}
to saying that $\Vert 1 - x \Vert_{(A^1)^{**}} = \Vert e - x \Vert_{A^{**}}$ for all $x \in A^{**}$,
unless the last quantity is $< 1$ in which case $\Vert 1 - x \Vert_{(A^1)^{**}} = 1$.
Here  $e$ is the identity  for $A^{**}$ if it has one,
otherwise
it is a mixed identity of norm 1.   A result of Effros and Ruan implies that  approximately unital
operator algebras are $M$-approximately unital (see e.g.\ \cite[Theorem 4.8.5 (1)]{BLM}).    Also, all
unital Banach algebras are  $M$-approximately unital.

We use states a lot in our work.   However for an approximately unital Banach algebra $A$ with cai $(e_t)$, the definition of `state' is problematic.  Although we have not noticed this discussed in the literature, there are several natural notions, and which is best seems to depend on the
situation.      For example: (i) \ a contractive functional
$\varphi$ on $A$ with $\varphi(e_t) \to 1$ for some fixed  cai $(e_t)$ for $A$,
(ii) \   a contractive functional
$\varphi$ on $A$ with $\varphi(e_t) \to 1$ for all cai $(e_t)$ for $A$,  and
(iii) \ a norm $1$ functional on $A$ that extends to a state on $A^1$,
where $A^1$ is the `multiplier unitization' above.   If $A$
satisfies a smoothness hypothesis then all these notions coincide  \cite[Lemma 2.2]{BOZ}, but this is not true
in general.   The $M$-approximately unital Banach algebras in the last paragraph are smooth in this 
sense.   Also,  if $e$ is a mixed identity for $A^{**}$
 then the statement $\varphi(e) = 1$ may depend on which mixed identity one considers.
 In this paper though
for simplicity, and because of its connections with the usual theory of numerical range and accretive operators,
we will take (iii) above as the definition of a {\em state} of $A$.    In \cite{BOZ} we also consider
some of the other variants above, and these will appear  below from time to time.
 We define  the state space $S(A)$ to be
the set of states
in the sense of (iii) above.    The quasistate space $Q(A)$ is $\{ t \varphi : t \in [0,1], \varphi \in S(A) \}$.
The numerical range of $x \in A$ is $W_A(x) = \{ \varphi(x) :  \varphi \in S(A) \}$.    As in \cite{BOZ}
we define ${\mathfrak r}_{A^{**}} = A^{**} \cap 
{\mathfrak r}_{(A^1)^{**}}$.    There is an unfortunate ambiguity with the latter notation here and in
\cite{BOZ} in the (generally rare) case that $A^{**}$ is unital.  It  should be 
stressed that in these papers ${\mathfrak r}_{A^{**}}$ should {\em not}, if $A^{**}$ is unital,  be confused with the 
real positive (i.e.\ accretive) elements in $A^{**}$.  It is shown in 
\cite[Section 2]{BOZ} that  these are the 
same if $A$ is an $M$-approximately unital Banach algebra, and in particular if $A$ is an approximately unital operator algebra.      It is easy to see that $A^{**} \cap 
{\mathfrak r}_{(A^1)^{**}}$ is contained in the accretive elements in $A^{**}$ if $A^{**}$ is unital,
but the other direction seems unclear in general.

Of course in the theory of $C^*$-algebras, positivity and the existence of positive approximate identities are crucial.
How does one get a `positive cai' in an algebra with cai?      We have several ways to do this.
First, for approximately unital operator algebras and for a large class of approximately unital 
Banach algebras (eg.\ the {\em scaled} Banach algebras defined in the next section; and we do not possess an example 
of a Banach algebra that is not scaled yet)
 we have a `Kaplansky density' result: $\overline{{\rm Ball}(A) \cap {\mathfrak r}_A}^{w*} = {\rm Ball}(A^{**}) \cap {\mathfrak r}_{A^{**}}$.   See Theorem \ref{Kap}
 below.    (We remark that although it seems not to be well known, the most common variants of the 
usual Kaplansky density theorem for a $C^*$-algebra $A$
do follow quickly from the weak* density of 
${\rm Ball}(A)$ in ${\rm Ball}(A^{**})$, if 
one constructs $A^{**}$ carefully.)
 If $A^{**}$  has a real positive mixed identity $e$ of norm 1, then one 
can then get a real positive cai by approximating $e$ by elements of ${\rm Ball}(A) \cap {\mathfrak r}_A$.    See
  Corollary \ref{solos}.  
A similar argument allows one to deduce the second assertion in the following result from the first (one also
uses  the fact that in an $M$-approximately unital Banach algebra $\Vert 1 - 2e \Vert \leq 1$ for a mixed identity 
of norm $1$ for $A^{**}$):

\begin{theorem} \label{Mcai} {\rm \cite{BOZ,BRI,Read}} Let $A$ be an
$M$-approximately unital Banach algebra, for example any operator algebra.   Then
 ${\mathfrak F}_A$ is weak* dense in ${\mathfrak F}_{A^{**}}$.
Hence $A$ has a cai in $\frac{1}{2}{\mathfrak F}_A$.   \end{theorem}

Applied to approximately unital
operator algebras (which as we said are all
$M$-approximately unital) the last assertion of
Theorem \ref{Mcai} becomes  
Read's theorem from \cite{Read}.   See also \cite{Bnpi, BRord}
for other proofs of the latter result.

\begin{remark}  For the conclusion that 
${\mathfrak F}_A$ is weak* dense in ${\mathfrak F}_{A^{**}}$
one may relax the $M$-approximately unital hypothesis to the following
much milder condition: $A$ is  approximately unital  and given $\epsilon > 0$ there exists a $\delta > 0$
such that if  $y \in A$ with $\Vert 1 - y \Vert < 1 + \delta$ then there is a $z \in A$
with $\Vert 1 - z \Vert = 1$ and $\Vert y - z \Vert < \epsilon$.  
 Here $1$ denotes
the identity of any unitization of $A$.  This follows from the proof of \cite[Theorem 5.2]{BOZ}.
For example, $L^1(\Rdb)$ satisfies this condition with $\delta = \epsilon$.  
\end{remark}

Another approach to finding a `real positive cai' under a countability condition from \cite[Section 2]{BOZ}
uses a slight variant of the `real positive' definition.  Namely for a fixed cai ${\mathfrak e} = (e_t)$  for $A$ define
 $S_{\mathfrak e}(A) = \{ \varphi \in {\rm Ball}(A^*) :  
\lim_t \, \varphi(e_t) = 1 \}$ (a subset of $S(A)$).   Define  ${\mathfrak r}^{\mathfrak e}_A = \{ x \in A : {\rm Re} \, \varphi (x) \geq 0 \; {\rm for \, all} \; \varphi 
\in S_{\mathfrak e}(A) \}$.    If we multiply these states by numbers in $[0,1]$, we get 
the associated quasistate space $Q_{\mathfrak e}(A)$.     Note that ${\mathfrak r}^{\mathfrak e}_A$ contains
${\mathfrak r}^{\mathfrak e}_A$.  On the other hand, 
\cite[Theorem 6.5]{BOZ} (or a minor variant of the proof of it) shows that if $A^{**}$ 
is unital then 
${\mathfrak r}^{\mathfrak e}_A$ is never contained in  ${\mathfrak r}_{A^{**}}$ (or in the 
accretive elements in $A^{**}$)  unless ${\mathfrak r}^{\mathfrak e}_A = {\mathfrak r}^{\mathfrak e}_A$.

\begin{theorem} \label{seqcai} {\rm \cite{BOZ}}
  A Banach algebra $A$ with a sequential cai ${\mathfrak e}$ and with $Q_{\mathfrak e}(A)$ weak* closed, has a sequential cai in ${\mathfrak r}^{\mathfrak e}_A$.    \end{theorem}  

\begin{proof}   We give the main idea of the proof in \cite{BOZ}, and a few more details for the 
first step.   Suppose that $K$ is a compact space
and  $(f_n )$ is a bounded sequence
in $C(K,\Rdb)$, such that $\lim_n f_n(x)$ exists
for every $x \in K$ and is non-negative.   Claim: 
for every $\epsilon>0$, there is a function $f \in {\rm conv} \{ f_n \}$
such that $f \geq -\epsilon$ on $K$.   Indeed if this were not true, then there exists
an $\epsilon > 0$ such that for all 
$f \in {\rm conv} \{ f_n \}$ there is a point $x$  in $K$ with $f(x) < - \epsilon$.
Moreover, for all $g \in \overline{{\rm 
conv}} \{ f_n \}$, if $f \in {\rm conv} \{ f_n \}$ with $\Vert f-g \Vert < \frac{\epsilon}{4}$,
there is a point $x$  in $K$ with $g(x) < - \frac{3 \epsilon}{4}$.   So 
$A = \overline{{\rm 
conv}} \{ f_n  \}$ and $C = C(K)_+$
are clearly disjoint.   Moreover, it is well known that convex sets $E, C$ in an LCTVS can be strictly separated iff
$0 \notin \overline{E-C}$, and this is clearly the case for us here.  So there 
is a continuous functional $\psi$ on $C(K,\Rdb)$ and scalars $M,N$ 
with $\psi(g) \leq M < N \leq \psi(h)$ for all $g  \in A$ and 
$h \in C$.   Since $C$ is a cone we may take $N = 0$.  By the Riesz--Markov theorem there is
a   Borel probability measure $m$ such that
$\sup_n \int_K  f_n \, dm < 0$.  This is a contradiction and proves the Claim,
 since $\lim_n \, \int_K \, f_n \, dm \geq 0$ by 
Lebesgue's dominated  convergence theorem.   

Now let $K = Q_{\mathfrak e}(A)$ and $f_n(\psi) = {\rm Re} \, \psi(e_n)$ where ${\mathfrak e} = (e_n)$.
We have $\lim_n f_n \geq 0$ pointwise on $K$, so by the last paragraph for any $\epsilon>0$ a convex combination of the 
$f_n$ is always $\geq -\epsilon$ on $K$.     By a standard geometric series type argument we can replace $\epsilon$ with $0$ here, so that we have a real positive element, and with more care this convex combination may be taken to be a  generic element in a cai.
 \end{proof}

Finally, we state a `new' result, which will be proved in Corollary \ref{solos2} below (this result was referred to incorrectly   in the published version of \cite{BOZ} 
as `Corollary 3.4' of the present paper).  

\begin{corollary} \label{solos2a}  If $A$ is an approximately unital 
Banach algebra with a cai ${\mathfrak e}$ such that $S(A) = S_{\mathfrak e}(A)$, and such that the quasistate space $Q(A)$ is weak* closed, then
$A$ has a cai in ${\mathfrak r}_{A}$.  \end{corollary} 

We remark  that we have no example of an approximately unital 
Banach algebra  where  $Q(A)$ is not weak* closed.    In particular,  we have found that 
commonly encountered algebras have this
property.

\section{Order theory in the unit ball} \label{pomoo}

In the spirit of the quotation starting Section \ref{kad} 
we now discuss generalizations of
well known order-theoretic properties
of the unit ball of a $C^*$-algebra and its dual. 
Some of these results also may be viewed as new relations between an operator algebra and a $C^*$-algebra that it 
generates.    There are interesting connections to the classical theory 
of ordered linear spaces (due to Krein, Ando, Alfsen, etc) as found e.g.\ in the first chapters of \cite{AE}.
In addition to striking parallels, some of this classical theory can be applied directly.  Indeed several results from 
\cite{BOZ} (some of which are mentioned below, see e.g.\ the proof of
Theorem   \ref{brcof4}) are proved by appealing to results in that theory. 
See also
\cite{BRord} for more connections if the algebras are in addition operator algebras.  

The ordering induced by ${\mathfrak r}_A$ is obviously
$b \preccurlyeq a$ iff $a-b$ is accretive (i.e.\ numerical range in right half plane).  If $A$ is an operator algebra this happens when 
${\rm Re}(a - b) \geq 0$.

 \begin{theorem}   \label{brcof} {\rm \cite{BRord}} 
 If an approximately unital  operator algebra $A$ generates a $C^*$-algebra
$B$, then $A$ is {\em order cofinal} in $B$.  That is, given $b \in B_+$ there exists $a \in A$ with $b \preccurlyeq a$.   Indeed one can do this with
$b \preccurlyeq a \preccurlyeq \Vert b \Vert + \epsilon$.
 Indeed one can do this with
$b \preccurlyeq C e_t \preccurlyeq \Vert b \Vert + \epsilon$, for a nearly positive cai $(e_t)$ for $A$ and a constant $C 
> 0$.
\end{theorem}

This and the next result  are trivial if $A$ unital.  

\begin{theorem}   \label{brcof2}  {\rm \cite{BRord}}
  Let  $A$ be an operator algebra which generates a $C^*$-algebra
$B$, and let ${\mathcal U}_A = \{ a \in A : \Vert a \Vert < 1 \}$.  The following are equivalent:
\begin{itemize} \item [(1)]   $A$ is approximately unital.
 \item [(2)]  For any positive $b \in {\mathcal U}_B$ there exists  $a \in {\mathfrak c}_A$
with $b \preccurlyeq a$.
 \item [(2')]  Same as {\rm (2)}, but also $a \in \frac{1}{2}  {\mathfrak F}_A$ and nearly positive.
\item [(3)]   For any pair
$x, y \in {\mathcal U}_A$ there exist   nearly positive
$a \in \frac{1}{2}  {\mathfrak F}_A$
with $x \preccurlyeq a$ and $y \preccurlyeq a$.
\item [(4)]    For any $b \in {\mathcal U}_A$  there exist   nearly positive
$a \in \frac{1}{2}  {\mathfrak F}_A$
with $-a \preccurlyeq b \preccurlyeq a$.
\item [(5)] For any $b \in {\mathcal U}_A$  there exist
$x, y \in  \frac{1}{2}  {\mathfrak F}_A$ 
with $b = x-y$.
\item [(6)]  ${\mathfrak C}_A$ is a generating cone (that is, $A = {\mathfrak C}_A - {\mathfrak C}_A$).
\end{itemize}
\end{theorem}

In any operator algebra $A$ it is true that ${\mathfrak C}_A - {\mathfrak C}_A$
is a closed subalgebra of $A$.  It is the biggest approximately unital subalgebra of $A$,
and it happens to also be a HSA in $A$ \cite{BRII}.  We do not know
%QQ
 if this is true for  Banach algebras.

For `nice' Banach algebras $A$ the cone ${\mathfrak C}_A$  has some of the pleasant order properties in 
items (3)--(6) in Theorem \ref{brcof2}.   See \cite[Section 6]{BOZ} for various variants on this theme.  
The following is a particularly clean case:

\begin{theorem}   \label{brcof3}  {\rm \cite[Section 6]{BOZ}}
  Let  $A$ be an $M$-approximately unital  Banach algebra.  Then 
\begin{itemize} \item [(1)]   For any pair
$x, y \in {\mathcal U}_A$ there exist  
$a \in \frac{1}{2}  {\mathfrak F}_A$
with $x \preccurlyeq a$ and $y \preccurlyeq a$.
\item [(2)]    For any $b \in {\mathcal U}_A$  there exist  
$a \in \frac{1}{2}  {\mathfrak F}_A$
with $-a \preccurlyeq b \preccurlyeq a$.
\item [(3)] For any $b \in {\mathcal U}_A$  there exist
$x, y \in  {\mathcal U}_A \cap \frac{1}{2}  {\mathfrak F}_A$
with $b = x-y$.
\item [(4)]  ${\mathfrak C}_A$ is a generating cone (that is, $A = {\mathfrak C}_A - {\mathfrak C}_A$).
\end{itemize}
\end{theorem}

During the writing of the present paper we  saw the following  improvement of part of
Corollaries 6.7 and 6.8, and on some of 6.10 in the submitted version of 
the paper \cite{BOZ}.    At the galleys stage of that paper we incorporated those advances,
but unfortunately  slipped up  in one proof.   The correct version is as below.

\begin{theorem}   \label{brcof4}  {\rm \cite[Section 6]{BOZ}}  If a Banach algebra $A$ has a 
cai ${\mathfrak e}$ and 
satisfies that $Q_{{\mathfrak e}}(A)$ is weak* closed, then 
{\rm (1)--(4)} in the last theorem hold, with $\frac{1}{2}  {\mathfrak F}_A$
replaced by ${\rm Ball}(A) \cap {\mathfrak  r}^{{\mathfrak e}}_A$, and $\preccurlyeq$ replaced
by the linear ordering defined by the cone ${\mathfrak r}^{{\mathfrak e}}_A$, and  ${\mathfrak C}_A$ 
replaced by  ${\mathfrak r}^{{\mathfrak e}}_A$.     One may drop
the three superscript  ${\mathfrak e}$'s in the last line if in 
addition $S(A) = S_{{\mathfrak e}}(A)$.  \end{theorem}

\begin{proof}    Lemma 2.7 (1) in \cite{BOZ} implies  that if $Q_{{\mathfrak e}}(A)$ is weak* closed, 
then the `dual cone' in $A^{*}$  of ${\mathfrak  r}^{{\mathfrak e}}_A$ is $\Rdb^+  \, S_{{\mathfrak e}}(A)$.
By the remark before  \cite[Proposition 6.2]{BOZ} a similar fact holds for the real dual cone.
 Since $\Vert \varphi \Vert = 1$ for states and for their real parts, 
the norm on the real dual cone is additive.   This is known to imply, by the theory of ordered linear spaces
\cite[Corollary 3.6, Chapter 2]{AE}, that the open ball of $A$ is a directed set.
So for any pair
$x, y \in {\mathcal U}_A$ there exist $z \in {\mathcal U}_A$ with $x \preccurlyeq_{\mathfrak e} z$ and
$y  \preccurlyeq_{\mathfrak e} z$.  Applying this again to $z, -z$ there exists
$w \in {\mathcal U}_A$ with $\pm z \preccurlyeq_{\mathfrak e} w$.   This implies that $\frac{w \pm z}{2} 
\in {\mathfrak  r}^{{\mathfrak e}}_A$, and $z  \preccurlyeq_{\mathfrak e}  a$ 
where $a = \frac{w + z}{2}$.  This proves (1).  Applying (1) to $b , -b$ we get (2).   Setting 
$x = \frac{a + b}{2}, y = \frac{a - b}{2}$ for $a, b$ as in (2), we get (3) and hence (4).
The final assertion is then obvious since if $S(A) = S_{{\mathfrak e}}(A)$ then 
${\mathfrak  r}_A = {\mathfrak  r}^{{\mathfrak e}}_A$ and $ \preccurlyeq_{\mathfrak e}$ is just
$ \preccurlyeq$.  
\end{proof}

  Recall that the positive part of the  open unit ball ${\mathcal U}_B$ of a $C^*$-algebra $B$
is a directed set, and indeed is a  net which is a positive cai for $B$.     The first part of this statement is generalized 
by 
Theorems \ref{brcof2} (3), \ref{brcof3} (1),
and \ref{brcof4} (1).   The following generalizes the second part of the statement
to operator algebras:

\begin{corollary} \label{ballcai}  {\rm \cite{BRord}}   If $A$ is an approximately unital  operator algebra, then
${\mathcal U}_A \cap \frac{1}{2} {\mathfrak F}_A$ is a directed set in the $\preccurlyeq$ ordering,
and with this ordering ${\mathcal U}_A \cap \frac{1}{2} {\mathfrak F}_A$  is an increasing
cai for $A$.  \end{corollary}

We do not know if the second part of the last result is true for any other classes of Banach algebras.  
%QQ

We say a Banach algebra $A$ is {\em scaled} if every real positive    
linear map into the scalars is a nonnegative multiple of a state.   Of course it is well known that $C^*$-algebras are scaled.    Somewhat surprisingly, we do not know 
%QQ
of an approximately unital 
 Banach algebra that is not scaled, and certainly all commonly encountered Banach algebras seem to be scaled.   
Unital  Banach algebras are scaled by e.g. an argument in the proof of  \cite[Theorem 2.2]{Mag}.

\begin{theorem} \label{nnms}  {\rm \cite{BRord,BOZ}}\      If $A$ is an approximately unital operator algebra, or more generally an $M$-approximately unital Banach algebra, then 
$A$ is scaled. 
  \end{theorem}

For operator algebras, the last result implies Read's theorem mentioned earlier.    

\begin{proposition} \label{scal} {\rm \cite{BOZ}}\   If $A$ is a nonunital  approximately unital 
 Banach algebra, then the following are equivalent:
\begin{itemize}
\item [(i)]    $A$ is scaled.
\item [(ii)]     $S(A^1)$ is the convex hull of the 
trivial character $\chi_0$ and the set of states on $A^1$ extending states of $A$.  
\item [(iii)]   The quasistate space $Q(A) = \{ \varphi_{|A} : \varphi \in S(A^1) \}$.  
\item [(iv)]   $Q(A)$ is convex and weak* compact.
\end{itemize}   If these hold then  $Q(A) = \overline{S(A)}^{{\rm w*}}$, and the numerical range satisfies
$$\overline{W_{A}(a)}  =  {\rm conv} \{ 0, W_{A}(a) \} =  W_{A^1}(a), \qquad a \in A.$$
\end{proposition}

\begin{theorem}[Kaplansky density type result] \label{Kap} 
  If $A$ is a scaled approximately unital 
Banach algebra then ${\rm Ball}(A) \cap {\mathfrak r}_A$ 
is weak* dense in the unit ball of ${\mathfrak r}_{A^{**}}$.  \end{theorem}

The last result is from \cite{BOZ}, although some operator algebra
variant was done earlier with Read.      

\begin{corollary} \label{solos}  If $A$ is a scaled approximately unital 
Banach algebra then
$A$ has a cai in ${\mathfrak r}_{A}$ iff $A^{**}$ 
has a mixed identity $e$ of norm $1$  in ${\mathfrak r}_{A^{**}}$, or equivalently
with 
$\Vert 1_{A^1} - e \Vert \leq 1$.  \end{corollary} \begin{proof} 
This is proved in \cite[Proposition 6.4]{BOZ}, relying on earlier 
results there, except for parts of the last assertion.    For the 
remaining part,  if $A$ has a cai in ${\mathfrak r}_{A}$ then 
a cluster point of this cai is a mixed identity of norm $1$,
and it is in ${\mathfrak r}_{(A^1)^{**}}$ since the latter
is weak* closed and contains ${\mathfrak r}_{A}$.  However by
a result from
\cite{BOZ} (see  Lemma \ref{lump} below), an idempotent is in ${\mathfrak r}_{(A^1)^{**}}$
iff it is in ${\mathfrak F}_{(A^1)^{**}}$.    
\end{proof}

\begin{corollary} \label{solos2}  If $A$ is a scaled approximately unital 
Banach algebra with a cai ${\mathfrak e}$ such that $S(A) = S_{\mathfrak e}(A)$ then
$A$ has a cai in ${\mathfrak r}_{A}$.    
 \end{corollary} \begin{proof} 
Let $e$ be any weak* limit point of ${\mathfrak e}$.   Clearly
$\varphi(e) = 1$ for all $\varphi \in S_{\mathfrak e}(A) = S(A)$.     If $\varphi \in S((A^1)^{**})$
then its restriction to $A^1$ is in $S(A^1)$, hence $\varphi(e) \geq 0$ by the last line 
and  Proposition \ref{scal}.
So $e \in A^{**} \cap {\mathfrak r}_{(A^1)^{**}} = {\mathfrak r}_{A^{**}}$, and so the 
result  follows from our Kaplansky density type theorem
in the form of its Corollary \ref{solos}.     \end{proof}  

The class of algebras $A$ in the last Corollary is the same as the class in the last line of the statement
of Theorem \ref{brcof4}.    Thus for such algebras, {\rm (1)--(4)} in Theorem  \ref{brcof3} hold, with $\frac{1}{2}  {\mathfrak F}_A$
replaced by ${\rm Ball}(A) \cap {\mathfrak  r}_A$, and  ${\mathfrak C}_A$ 
replaced by  ${\mathfrak r}_A$.   In particular, ${\mathfrak r}_A$ spans $A$.  

\section{Positivity and roots in Banach algebras} \label{poroo}

As we said in the Introduction, this section and the next have several purposes:  We will describe results from 
our other papers (particularly \cite{BOZ}, which generalizes some parts of the 
earlier work) connected to the work of Kadison summarized in Section \ref{kad}, but we will also 
 restate the results  from several sections of \cite{BOZ} in the more general setting of Banach algebras 
with no kind of approximate identity.   Also we will give a detailed discussion of roots (fractional powers) in relation
to our positivity (see also \cite{BRI, BRII,BBS,BRord}  for results not covered here).

Thus let $A$ be a Banach algebra without a cai, or without any kind of bai.  
If $B$ is any unital Banach algebra isometrically 
containing $A$ as a subalgebra, for example any unitization of $A$,
we define  $${\mathfrak F}^B_A = \{ a \in A : \Vert 1_B - a \Vert
\leq 1 \},$$ and write ${\mathfrak r}^B_A$ for the set of $a \in A$ whose 
numerical range in $B$ is contained in the right half plane.    These sets are closed and convex.
 Also we define
  $${\mathfrak F}_A \, = \, \cup_B \, {\mathfrak F}^B_A, \; \; \; \;  \; \; \; 
{\mathfrak r}_A  \, = \, \cup_B \,  {\mathfrak r}^B_A,$$
the unions taken over all unital Banach algebras $B$
containing $A$.    Unfortunately it is not clear to us that
 ${\mathfrak F}_A$ and ${\mathfrak r}_A$ are always convex, 
%QQ
which is needed in most of 
Section \ref{AKba} below (indeed we often need them closed too there), and so we 
will have to use ${\mathfrak F}^B_A$ and ${\mathfrak r}^B_A$ there instead.  
Of course we could fix this problem by defining 
${\mathfrak F}_A = \cap_B \, {\mathfrak F}^B_A$ and ${\mathfrak r}_A  = \cap_B \,  {\mathfrak r}^B_A,$
 the intersections taken over all $B$ as above.  If we did this
then we could remove the superscript $B$ in all 
results in Section \ref{AKba} below; this would look much cleaner but
may be less useful in practice.

Obviously  ${\mathfrak F}_A$ and ${\mathfrak r}_A$ are 
convex and closed if there is an  unitization $B_0$ of $A$
such that ${\mathfrak F}^{B_0}_A = {\mathfrak F}_A$ (resp.\ ${\mathfrak r}^{B_0}_A = {\mathfrak r}_A$).
  This happens if $A$ is
an operator algebra because  then there is a unique unitization by a theorem of Ralf Meyer  (see
 \cite[Section 2.1]{BLM}).  
The following is another case when this happens.

\begin{lemma} \label{exu}   Let $A$ be a nonunital Banach algebra.  
 \begin{itemize} \item [(1)]   Suppose that there exists a `smallest'  unitization norm on $A \oplus \Cdb$.
That is, there exists a smallest norm  on $A \oplus \Cdb$
making it a normed algebra with product $(a,\lambda) (b,\mu) = (ab + \mu a + \lambda b,\lambda \mu)$, and satisfying 
$\Vert (a,0) \Vert = \Vert a \Vert_A$ for $a \in A$. Let  $B_0$ be $A \oplus \Cdb$ with this smallest norm.   
  Then 
 ${\mathfrak F}^{B_0}_A = {\mathfrak F}_A$ and ${\mathfrak r}^{B_0}_A = {\mathfrak r}_A$.    
 \item [(2)]  Suppose that  the   left regular representation embeds $A$ isometrically in $B(A)$.    
(This is the case for example if $A$ is  approximately unital.)     Define $B_0$ to be the span in $B(A)$ of $I_A$ and the isometrically  embedded
copy of $A$.   This has the smallest norm of any unitization of $A$.   Hence 
 ${\mathfrak F}^{B_0}_A = {\mathfrak F}_A$ and ${\mathfrak r}^{B_0}_A = {\mathfrak r}_A$.  
\end{itemize}  
    \end{lemma}

\begin{proof}  
   If $B$ is any unital Banach algebra
containing $A$, and $a \in {\mathfrak F}^B_A$ then $a \in {\mathfrak F}^{B_0}_A$.   So 
${\mathfrak F}^{B_0}_A = {\mathfrak F}_A$.    
 A similar argument 
shows that  ${\mathfrak r}^{B_0}_A = {\mathfrak r}_A$, using  Lemma \ref{chaccr} (2), namely that 
$${\mathfrak r}^B_A = \{ a \in A : \Vert 1_B - t a \Vert
\leq 1 + t^2 \Vert a \Vert^2 \; \textrm{for all} \; t \geq 0 \}.$$

(2)\   The first assertion here is well known and simple: If $B$ is any unital Banach algebra
containing $A$  note that $$\Vert a + \lambda 1_{B_0} \Vert_{B_0}
= \sup \{ \Vert (a + \lambda 1) x  \Vert_A : x \in {\rm Ball}(A) \} 
= \sup \{ \Vert (a + \lambda 1_B) x  \Vert_B : x \in {\rm Ball}(A) \},$$
so that $\Vert a + \lambda 1_{B_0} \Vert_{B_0} \leq \Vert a + \lambda 1_{B} \Vert_{B} .$
  So  (1) holds.  
\end{proof}

\begin{proposition} \label{aufr}  If $A$ is a  nonunital subalgebra of a unital
Banach algebra $B$, and if $C$ is a  subalgebra of $A$, then  ${\mathfrak F}^B_{C}= C \cap {\mathfrak F}^B_{A}$
and ${\mathfrak r}^B_{C}= C \cap {\mathfrak r}^B_{A}$.  
\end{proposition}

We now discuss   {\em roots} (that is, $r$'th powers  
for  $r \in [0,1]$) in a subalgebra $A$ of a unital Banach algebra $B$.   
Actually, we only discuss the principal root (or power); we recall that the {\em principal}  $r$th power, for $0 < r < 1$,   
is the one whose spectrum is contained in a sector $S_\theta$ of angle $\theta < 2 r \pi$.
There are several ways to define these that we are aware of.  We will
review these and show that they are the same.    As far as we know, Kelley and Vaught \cite{KV} were the first to define 
the square root of elements of ${\mathfrak F}_A$, but their argument works for  $r$'th powers  
for  $r \in [0,1]$.   If $\Vert 1 - x \Vert \leq 1$, define 
$$x^{r} = \sum_{k = 0}^\infty \, {r \choose k} (-1)^k (1-x)^k \; , \qquad 
r > 0.$$
For $k \geq 1$ the sign of ${r \choose k} (-1)^k$ is
always negative, and $\sum_{k = 1}^\infty \,
{r \choose k} (-1)^k = -1$.    Thus the series above converges absolutely, hence converges in $A$.
Indeed it is now easy to see that the series given for $x^r$ 
 is a norm limit of polynomials in $x$ with no constant term.    Using
 the Cauchy product formula in Banach algebras in a standard way,
one deduces that  
$(x^{\frac{1}{n}})^n = x$ for any positive integer $n$.

\begin{proposition}[Esterle] \label{clnth}  If   $A$ is a Banach algebra then
  ${\mathfrak F}_A$ 
is  closed under $r$'th powers  
for any $r \in [0,1]$.  
\end{proposition} 

\begin{proof}  Let $x \in A \cap {\mathfrak F}_B$ where $B$ is a unital Banach algebra
containing $A$.  We have   $1_B - x^{r} = \sum_{k=1}^\infty\, {r \choose k} (-1)^k \, (1_B-x)^k$, which is a convex combination
in Ball$(B)$.   So $x^{r} \in A \cap {\mathfrak F}_B \subset {\mathfrak F}_A$.  
\end{proof} 

From  \cite[Proposition 2.4]{Est} if $x \in {\mathfrak F}_A$ then we also have $(x^t)^r = x^{tr}$ for $t \in [0,1]$ and 
any real $r$.  
One cannot use the usual Riesz functional  calculus to define $x^r$ if $0$ is in the spectrum of $x$, since such $r$'th powers are badly behaved  at $0$.    However if $0$ is in the spectrum of $x$, and $x \in {\mathfrak r}_A^B$, one may define 
$x^r = \lim_{\epsilon \to 0^+} \; (x + \epsilon 1_B)^r$ where the latter is the  $r$'th power according to the Riesz functional  calculus.    We will soon see that this limit exists  and lies in $A$, and then it follows that it is independent of the particular 
unital algebra $B$ containing $A$ as a subalgebra (since all unitization norms for $A$ are equivalent).  
A second  way to define $r$'th powers  
for  $r \in [0,1]$) in Banach algebras is found in \cite{LRS}, following the ideas in Hilbert space operator case from the Russian  
literature from the 50's \cite{MP}.   Namely, suppose that $B$ is a unital Banach algebra containing $A$ as a subalgebra,
and $x \in A$ with numerical range in $B$ excluding all negative numbers.
Since the numerical range is convex, it follows  that this numerical range  is in fact contained in a sector (i.e.\ a cone in the complex plane with vertex at $0$) of angle $\leq \pi$.   Since this is the case we are interested in, we will assume that 
the numerical range of $x$  is in the closed right half plane.  (This is usually not really any loss of generality, since $x$ and hence the just mentioned cone can be `rotated' to ensure this.)   Thus the numerical range of $x$  is contained inside a semicircle, namely the one containing the right half of the circle center $0$  radius $R > 0$.    We enlarge this semicircle to a slightly larger `slice' of this circle of radius $R$; thus let $\Gamma$ be the positively oriented contour which is 
symmetric about the $x$-axis, and is composed of an arc of the circle slightly bigger that the right half of the circle, and two line segments  which connect zero with the arc.  Let $\Gamma_\epsilon$ be $\Gamma$ but with points removed that are distance less than $\epsilon$ to the origin.
One defines $x^t$ to be the limit as $\epsilon \to 0$ of $\frac{1}{2 \pi i} \, \int_{\Gamma_\epsilon} \, \lambda^t (\lambda
1_B  - x)^{-1} \, d \lambda$.     The latter integral lies in $A + \Cdb 1_B$, by the usual facts about such integrals.   If $A$ is
nonunital and $\chi_0$ is the character on  $A + \Cdb 1_B$ annihilating $A$ then $\chi_0(x^t)$ is the limit of $\frac{1}{2 \pi i} \, \int_{\Gamma_\epsilon} \, \lambda^t (\lambda
1  - \chi_0(x))^{-1} \, d \lambda$, which is $\frac{1}{2 \pi i} \, \int_{\Gamma}  \, \lambda^t \,  d \lambda = 0$.   So $x^t \in A$.
Note that $x^t$ is  independent of 
the particular unitization $B$ used,  
using the fact that all unitization norms are equivalent.    If in addition $x$ is invertible then $0 \notin {\rm Sp}_B(x)$, so that we can replace $\Gamma$ by a curve that stays to one side of $0$, so that $x^r$ is the  $r$th power of $x$ as given by the Riesz functional calculus.   In fact it is shown in \cite[Proposition 3.1.9]{LRS} that $x^r = \lim_{\epsilon \to 0^+} \, (x+\epsilon 1_B)^r$ for $t > 0$, giving the equivalence with the definition at the start of this discussion.    In addition we 
now see, as we discussed earlier, that the latter limit exists, lies in $A$, and is independent of $B$.    By  \cite[Corollary 1.3]{LRS}, the $r$th power  function is continuous on 
${\mathfrak r}^B_{A}$, for any  $r \in (0,1)$.    Principal nth roots of accretive elements 
are unique, for any positive integer $n$ (see \cite{LRS}).

A final way to define $r$'th powers  $x^r$
for  $r \in [0,1]$ and $x \in  {\mathfrak r}_A$,  is via the functional  calculus for sectorial operators \cite{Haase}
(see also e.g.\ \cite[IX, Section 11]{Yos} for some of the origins
of this approach).   Namely, if $B$ is a unitization of $A$
(or a  unital Banach algebra containing $A$ as a subalgebra) and $x \in  {\mathfrak r}_A^B$,
  view $x$ as an operator on $B$ by left multiplication.  This is sectorial of angle $\leq \frac{\pi}{2}$,  and so we can use the theory of roots (fractional powers) from 
e.g.\ \cite[Section 3.1]{Haase} (see also \cite{NF}).  
Basic properties of such powers include: $x^s x^t = x^{s+t}$ and $(cx)^t = c^t x^t$,  for 
positive scalars $c,s,t$, and $t \to x^t$ is continuous.      There are very many more in e.g.\ \cite{Haase}.
Also  \cite[Proposition 3.1.9]{Haase} shows that $x^r = \lim_{\epsilon \to 0^+} \, (x+\epsilon I)^r$ for $r > 0$, the latter power with respect to the 
usual Riesz functional calculus.      It is easy to see from the last fact that the definitions of $x^r$ given in this paragraph and in the last paragraph coincide if $x \in {\mathfrak r}_A$ and $r  >0$; so that again $x^r$ is in (the copy inside $B(B)$ of) $A$.    Another formula we have  occasionally found useful is $x^r =  \frac{\sin(t \pi)}{\pi} \int_0^\infty \, s^{r- 1} \, (s + x)^{-1} x \, ds$, the Balakrishnan formula  
(see e.g.\ \cite{Haase,Yos}).

We now  show that if 
$x \in {\mathfrak F}_A$ then  the definitions of $x^r$ given in 
the last paragraphs and in Proposition \ref{clnth} coincide,  if $r  > 0$. 
We may assume that $0 < r \leq 1$ and work in  a unital algebra $B$ containing $A$.
Let  $y = \frac{1}{1+\epsilon} (x + \epsilon 1_B)$.  Then $\Vert 1_B - y \Vert < 1$, and so 
$y^r$ as defined in the last paragraphs equals $\sum_{k = 0}^\infty \,
{r \choose k} (-1)^k (1_B-y)^k$ since 
both are easily seen to equal the  $r$th power of $y$ as given by the Riesz functional calculus.   However 
$\sum_{k = 0}^\infty \,
{r \choose k} (-1)^k (1-y)^k$ converges uniformly to  $\sum_{k = 0}^\infty \,
{r \choose k} (-1)^k (1-x)^k$, as $\epsilon \to 0^+$, since the norm of the difference of these two series is dominated by
$$\sum_{k = 1}^\infty \,
{r \choose k} (-1)^k \, (\frac{1}{1+\epsilon} - 1) \, \Vert (1-x)^k \Vert \; \leq \; \frac{\epsilon}{1 + \epsilon} \, \to \, 0 ,$$ 
using the fact that for $k \geq 1$ the sign of ${r \choose k} (-1)^k$ is
always negative.    Also, with the definition of powers in the  last paragraphs we have $y^r = (\frac{1}{1+\epsilon})^r = (x + \epsilon 1_B)^r \to x^r$ as $\epsilon \to 0^+$.   Thus  the definitions of $x^r$ given in 
the last paragraphs and in Proposition \ref{clnth} coincide in this case.

If $A$ is  a subalgebra of a unital 
Banach algebra $B$  then we define the ${\mathfrak F}$-transform on $A$ 
to be  ${\mathfrak F}(x) = x (1_B + x)^{-1} = 1_B - (1_B + x)^{-1}$ for $x \in {\mathfrak r}_A$.    This is
a relative of the well known Cayley transform in operator theory.
Note that  ${\mathfrak F}(x) 
\in {\rm ba}(x)$ by the basic theory of Banach algebras, and it does not depend on $B$, again because all unitization norms for $A$ are equivalent.  The inverse transform takes $y$ to $y(1_B -y)^{-1}$.   For operator algebras we 
have $\Vert {\mathfrak F}(x) \Vert \leq \Vert x \Vert$ and $\Vert \kappa(x) \Vert \leq \Vert x \Vert$  for 
$x \in {\mathfrak r}_A$.   For Banach algebras this is not true; for example on the group algebra of $\Zdb_2$.

Unless explicitly said to the contrary,
 the remaining results in this section are 
generalizations to general Banach algebras of results from \cite{BOZ}.  The main results here 
in the operator algebra case were proved earlier
by  the author and Read (some are much sharper in that setting).

\begin{lemma} \label{stam}  If $A$ is a  subalgebra of a unital Banach algebra  $B$ 
 then ${\mathfrak F}({\mathfrak r}^B_A) \subset {\mathfrak F}^B_A$
and  ${\mathfrak F}({\mathfrak r}_A) \subset {\mathfrak F}_A$.  \end{lemma}

 \begin{proof} This is because by a result of Stampfli 
and Williams \cite[Lemma 1]{SW},
$$\Vert 1_B - x (1_B + x)^{-1} \Vert  =  \Vert (1_B + x)^{-1}
\Vert \leq d^{-1} \leq 1$$ where $d$ is the distance from $-1$ 
to the numerical range in $B$ of $x$.  \end{proof}

The following was stated in \cite{BOZ} without proof details.

\begin{proposition} \label{poscirc}    If  $A$ is a unital Banach algebra 
and $x \in {\mathfrak r}_A$  and $\epsilon > 0$ then $x + \epsilon 1 \in C {\mathfrak F}_A$ where $C = 
\epsilon + \frac{\Vert x \Vert^2}{\epsilon}$.  \end{proposition}   \begin{proof}     
 We have 
$$\Vert 1 - C^{-1} (x + \epsilon 1) \Vert =  C^{-1}  \, \Vert (C - \epsilon)1  - x \Vert
= C^{-1} \,  \frac{\Vert x \Vert^2}{\epsilon}  \, \Vert 1 - \frac{\epsilon}{\Vert x \Vert^2} x \Vert.$$
By  Lemma \ref{chaccr} (2), this is dominated by $C^{-1}  \frac{\Vert x \Vert^2}{\epsilon} (1 + 
 \frac{\epsilon^2}{\Vert x \Vert^2}) = 1$.
 \end{proof}  

It follows easily from Proposition \ref{poscirc} that $\overline{\Rdb^+ {\mathfrak F}_A} = {\mathfrak r}_A$ if $A$ is
unital.   For nonunital algebras we use a different argument:

\begin{proposition} \label{whau}  If $A$ is a  subalgebra of a unital Banach algebra  $B$ then
$\overline{\Rdb^+ {\mathfrak F}^B_A} = {\mathfrak r}^B_A$ and 
$\overline{\Rdb^+ {\mathfrak F}_A} = {\mathfrak r}_A$.  
\end{proposition}  \begin{proof}  
If 
 $x \in {\mathfrak r}^B_A$ and $t \geq 0$, then $tx(1_B + tx)^{-1} \in {\mathfrak F}^B_A$ by Lemma \ref{stam}.
By elementary Banach algebra theory, $(1_B + tx)^{-1} \to 1_B$ as $t \searrow 0$.  So
 $x = \lim_{t \to 0^+} \, \frac{1}{t} \, tx(1_B + tx)^{-1}$, from which the results are clear.  \end{proof}

{\bf Remark.}  There is a numerical range lifting result that works in quotients  of Banach spaces with `identity' or of approximately unital Banach algebras, if one takes the quotient by an $M$-ideal (see \cite{CSSW} and the end of Section 8 in \cite{BOZ}).   This may  be viewed as a noncommutative Tietze theorem, as explained in the last paragraph
 of Section 8 in \cite{BOZ}.   As a consequence one can lift a real positive element  in such a
quotient $A/J$ to  a real positive in $A$.  This again is a generalization of a well known $C^*$-algebraic positivity 
results since as pointed out by Alfsen and Effros (and Effros and Ruan),
$M$-ideals in  a $C^*$-algebras (or, for that matter, in
an approximately unital operator algebra) are just the 
two-sided closed ideals (with a cai).   See e.g.\ \cite[Theorem 4.8.5]{BLM}.

\begin{lemma} \label{Bal} Let  $A$ be a Banach algebra.   If  $x \in {\mathfrak r}_A$,
 then  $||x^{t}|| \leq  \frac{2 \sin(t \pi)}{\pi t (1 - t)} \, \Vert x \Vert^t$ if $0 < t < 1$.   If $A$ is an operator algebra one may remove the $2$ in this estimate.  \end{lemma}

 To prove this and the next corollary:  by the above we may as well work in any unital Banach algebra containing $A$, and this case was done in
\cite{BOZ}.    In the operator algebra case a recent paper of Drury \cite{Drury} 
is a little more careful with the estimates
for the integral in the Balakrishnan formula mentioned above for $x^t$, and obtains
$$\Vert x^{t} \Vert \leq  \frac{\Gamma(\frac{t}{2}) \, \Gamma(\frac{1-t}{2})}{2 \sqrt{\pi} \Gamma(t) 
\Gamma(1 - t)}$$ if $0 < t < 1$
and $\Vert x \Vert \leq 1$.   Drury states this for matrices $x$, but the same proof works for operators
on Hilbert space.

\begin{lemma} \label{sin}
 There is a nonnegative sequence $(c_n)$ in $c_0$  such that for any  Banach algebra $A$,
and $x \in {\mathfrak F}_A$ or $x \in  {\rm Ball}(A) \cap {\mathfrak r}_A$, we have
$\Vert x^{\frac{1}{n}} x - x \Vert \leq c_n$ for all $n \in \Ndb$.
    \end{lemma}

\begin{remark}   If $A$ is a Banach  algebra and $x \in
{\mathfrak F}_A$ or  or $x \in  {\rm Ball}(A) \cap {\mathfrak r}_A$ is  nonzero then  
$\limsup_n \, \Vert x^{\frac{1}{n}} \Vert \leq 1$ 
is the same as saying $\lim_n \, \Vert x^{\frac{1}{n}} \Vert = 1$.
For $$\Vert x \Vert  \leq \Vert x^{\frac{1}{n}} x - x \Vert + \Vert x^{\frac{1}{n}} x \Vert
\leq c_n + \Vert x^{\frac{1}{n}} \Vert \Vert x \Vert, \qquad n \in \Ndb,$$ where $(c_n) \in c_0$
as in  Lemma \ref{sin}.
This property holds if $A$ is an operator algebra by the last assertion of
Lemma \ref{Bal}.   
\end{remark}

\begin{corollary} \label{hasbi}  A Banach algebra $A$ with a left bai 
(resp.\ right bai, bai) in 
${\mathfrak r}_A$ has a left bai 
(resp.\ right bai, bai)  in ${\mathfrak F}_A$.     And a similar statement holds
with ${\mathfrak r}_A$ and ${\mathfrak F}_A$ replaced by ${\mathfrak r}^B_A$ and ${\mathfrak F}^B_A$
for any unital Banach algebra $B$ containing $A$ as a subalgebra.
\end{corollary}

 \begin{proof}  If $(e_t)$ is a left bai  (resp.\ right bai, bai)  in
${\mathfrak r}_A$, let $b_t = {\mathfrak F}(e_t) 
\in {\mathfrak F}_A$.    By the proof in \cite[Corollary 3.9]{BOZ},
 $(b_t^{\frac{1}{n}})$ is a left bai (resp.\ right bai, bai) in
${\mathfrak F}_A$.  \end{proof}

\begin{remark}    If the bai in the last result is sequential, then so is
the one constructed in ${\mathfrak F}_A$.  \end{remark} 

We imagine that if a Banach algebra has a cai in ${\mathfrak r}_A$  then under mild conditions 
it has a cai in  ${\mathfrak F}_A$.    We give a couple of results along 
these lines, that are not in \cite{BOZ}.

\begin{corollary} \label{hasbi4}  Suppose that
 $A$ is a Banach  algebra with the property that there is a sequence 
$(d_n)$ of scalars with limit $1$ such that  
$\Vert x^{\frac{1}{n}} \Vert \leq d_n$ for all $n \in \Ndb$ and
 $x \in {\mathfrak F}_A$ 
(this is the case for operator algebras by Lemma {\rm \ref{Bal}}).  If $A$ has a
 left bai
(resp.\ right bai, bai) in
${\mathfrak r}_A$ then $A$ has a left cai
(resp.\ right cai, cai)  in ${\mathfrak F}_A$.  And a similar statement holds
with ${\mathfrak r}_A$ and ${\mathfrak F}_A$ replaced by ${\mathfrak r}^B_A$ and ${\mathfrak F}^B_A$
for any unital Banach algebra $B$ containing $A$ as a subalgebra.  
\end{corollary} 

\begin{proof}    For the first case, let $(f_s)_{s \in \Lambda}  = (b_t^{\frac{1}{n}})$ be the left bai 
in ${\mathfrak F}_A$ from Corollary \ref{hasbi}.     Note that $\Vert f_s \Vert \leq d_n$ 
and so it is easy to see that $\Vert f_s \Vert  \to 1$ by the Remark after Lemma \ref{sin}.
If there is a contractive subnet of $(f_s)$ 
we are done, so assume that there is no contractive subnet.   So for every $s \in \Lambda$ there is an
 $s' \geq s$ with $\Vert f_{s'} \Vert > 1$.   Let $\Lambda_0 = \{ s \in \Lambda :
\Vert f_{s} \Vert > 1 \}$.    A straightforward argument shows that $\Lambda_0$ is directed, and that
 $(f_s)_{s \in \Lambda_0}$ is a subset of  $(f_s)_{t \in \Lambda}$ which is
a  left bai 
 in ${\mathfrak F}_A$.    Then
  $(\frac{1}{\Vert f_s \Vert} \,  f_s )_{s \in \Lambda_0}$ is in ${\mathfrak F}_A$ since $\Vert f_s \Vert> 1$.  
So  $(\frac{1}{\Vert f_s \Vert} \,  f_s )_{s \in \Lambda_0}$  is a left cai in ${\mathfrak F}_A$.
The other cases are similar.
    \end{proof}

The hypothesis in the next result that 
$A^{**}$ is unital is, by \cite[Theorem 1.6]{BLP}, equivalent to there being a unique mixed
identity (we thank Matthias Neufang for this reference).   

\begin{proposition}   \label{asuni}   Let $A$ be a Banach algebra such that $A^{**}$ is unital and $A$ has
a real positive cai, 
or more generally suppose that there exists a real positive cai for $A$ and
a  bai for $A$ in  ${\mathfrak F}_A$ with  the same weak* limit.   Then $A$ has a cai in  ${\mathfrak F}_A$.   
This latter cai may be chosen to be sequential if in addition $A$ has a sequential bai.
\end {proposition}     \begin{proof}  That the second hypothesis is more general follows by Corollary
\ref{hasbi}
since a subnet of the ensuing bai for $A$ in  ${\mathfrak F}_A$ has a weak* limit.  
 Note that
if $(f_s)_{s \in \Lambda}$ is a bai in  ${\mathfrak F}_A$
with $\Vert f_s \Vert \to 1$ then either there is a subnet
of $(f_s)$ consisting of contractions, in which case
this subnet is a cai in  ${\mathfrak F}_A$, or
$\Lambda_0 = \{ s \in \Lambda : \Vert f_s \Vert \geq 1 \}$ is
a directed set and $(\frac{1}{\Vert f_s \Vert} \, f_s)_{s
 \in \Lambda_0}$ is a cai in  ${\mathfrak F}_A$.

Next, suppose that  $(e_t)$ is
a cai in ${\mathfrak r}_A$, and $(f_s)$ is a bai in  ${\mathfrak F}_A$
and they have the same weak* limit $f$.  By a re-indexing argument,
we can assume that they are indexed by the same directed set.
Then $e_t - f_t \to 0$ weakly in $A$.  If $E = \{ x_1, \cdots , x_n \}$
is a finite subset of $A$ define $F_{s,E}$ to be the subset $$\{ (e_t - f_t,
e_t x_1 - x_1 , x_1 e_t - x_1 , f_t x_1 - x_1 , x_1  f_t - x_1 ,
e_t x_2 - x_2 , \cdots , x_1  f_t - x_1) : t \geq s  \},$$
 of $A^{(4m+1)}.$
Since $(A^{(4m+1)})^*$ is the $1$ direct sum of $4m+1$ copies of $A^*$,
it is easy to see that $0$ is in the weak closure of $F_{s,E}$
(since $e_t - f_t \to 0$ weakly and $e_t x_k \to x_k$, etc).  Thus by
Mazur's theorem $0$ is in the norm closure of the convex hull of $F_{s,E}$. For each $n 
\in \Ndb$ there are a finite subset
$t_1, \cdots , t_K$ (where $K$ may depend on $n,s,E$),
and positive scalars
$(\alpha_k^{n,s,E})_{k=1}^K$ with sum $1$, such that if
$r_{n,s,E} = \sum_{k=1}^K \, \alpha_k^{n,s,E} e_{t_k}$ and $w_{n,s,E} = \sum_{k
=1}^K \, \alpha_k^{n,s,E}
f_{t_k}$,  then $\Vert r_{n,s,E} x_k - x_k \Vert,
 \Vert x_k r_{n,s,E} - x_k \Vert, \Vert x_k w_{n,s,E} - x_k \Vert,
\Vert w_{n,s,E} x_k - x_k \Vert,$ and $\Vert  r_{n,s,E} - w_{n,s,E} \Vert$,
are each less than $2^{-n}$ for all $k = 1, \cdots ,m$.
Note that $(r_{n,s,E})$ is then a cai in ${\mathfrak r}_A$,
and $(w_{n,s,E})$ is a bai in ${\mathfrak F}_A$.  Since
$r_{n,s,E} - w_{n,s,E} \to 0$ with $n$, it follows that
$\Vert w_{n,s,E} \Vert \to 1$ with $(n,E)$.   So as in the 
last paragraph
one may obtain from $(w_{n,s,E})$ a cai in ${\mathfrak F}_A$.

If we have a sequential cai in ${\mathfrak r}_A$ then it follows
from e.g.\ Sinclair's Aarnes-Kadison type theorem (see the lines after
Theorem \ref{AaKaSin}; alternatively one may use our 
 Aarnes-Kadison type theorem \ref{ottertoo} below) that $A = \overline{xAx}$
for some $x \in A$.  Given a cai $(f_t)$ in ${\mathfrak F}_A$,
choose $t_1 < t_2 < \cdots$ with $\Vert f_{t_k} x - x \Vert +
\Vert x f_{t_k}  - x \Vert < 2^{-k}$.  Then it is clear that
$(f_{t_k})$ is a sequential cai in ${\mathfrak F}_A$.
\end{proof}

\begin{remark}    It follows that under the conditions of the last result, one may improve \cite[Corollary 6.10]{BOZ}
in the way described after that result (using the fact in the remark after  \cite[Corollary 2.10]{BOZ}).  \end{remark} 

\begin{corollary} \label{hasb}  If $A$ is a Banach algebra then ${\mathfrak r}_A$ is closed under $r$th powers  
for any $r \in [0,1]$.   So is ${\mathfrak r}_A^B$  for any unital Banach algebra $B$ isometrically containing $A$ as a subalgebra.
\end{corollary} \begin{proof}    We saw in the 
proof of Proposition \ref{whau} 
that if $x \in {\mathfrak r}^B_A$ then
 $x = \lim_{t \to 0^+} \, \frac{1}{t} \, tx(1 + tx)^{-1}$, and $tx(1 + tx)^{-1} \in 
 {\mathfrak F}^B_A$.  Thus by \cite[Corollary 1.3]{LRS} we have that 
$x^r = \lim_{t \to 0^+} \, \frac{1}{t^r} \, (tx(1 + tx)^{-1})^r$ for $0 < r < 1$.  By Proposition \ref{clnth} and its proof,
the latter powers are in  $\Rdb^+ {\mathfrak F}^B_A$, so that $x^r \in \overline{\Rdb^+ {\mathfrak F}^B_A} = {\mathfrak r}_A^B \subset {\mathfrak r}_A$.  \end{proof}

In an operator algebra, if  $x$ is sectorial of angle $\theta \leq \frac{\pi}{2}$ then $x^t$ has  sectorial angle
$\leq t \theta$.  Indeed this is what allows us to produce `nearly positive elements', as discussed in Section 
 \ref{mnepai}.
The following, which we have not seen in the literature, may the best one has in a Banach algebra, and this 
disappointment means that some of the theory from \cite{BRI, BRII, BRord} will not generalize to Banach algebras.

\begin{corollary}  \label{sectt}  If $x$ is sectorial of angle $\theta \leq \frac{\pi}{2}$ in a unital  Banach algebra 
 then $x^t$ has  sectorial angle
$\leq t \theta + (1-t) \frac{\pi}{2}$.   
 \end{corollary}    \begin{proof}   This is Corollary \ref{hasb} if $\theta = \frac{\pi}{2}$.
Suppose that  $W_B(x) \subset S_\theta$.
Then $e^{\pm i (\frac{\pi}{2} - \theta)} x$ is accretive.  
  Hence  $(w x)^t$  is accretive where $w  = e^{\pm i (\frac{\pi}{2} - \theta)}$.   By \cite[Lemma 3.1.4]{Haase} with $f(z) = wz$  we have 
$(w x)^t = w^t \, x^t$.  So  $w^t x^t$ is accretive.
Reversing the argument above we see that $$W_B(x) \subset 
e^{i (\frac{\pi}{2} - \theta)} S_{\frac{\pi}{2}} \cap 
e^{- i (\frac{\pi}{2} - \theta)} S_{\frac{\pi}{2}} = 
S_{t \theta + (1-t) \frac{\pi}{2}}$$ as desired.  
\end{proof}

We learned the Hilbert space operator version of the last proof from Charles Batty.

\begin{proposition} \label{whba}   If $A$ is a  Banach algebra and $x \in {\mathfrak r}_A$
then ${\rm ba}(x) = {\rm ba}({\mathfrak F}(x))$, and so
$\overline{xA} =\overline{{\mathfrak F}(x) A}$.  \end{proposition}  \begin{proof}   We said earlier that ${\mathfrak F}(x)$ is in ${\rm ba}(x)$ and  is independent of  the particular unital 
Banach algebra containing $A$.   Thus this result follows from the unital case considered in \cite[Proposition 3.11]{BOZ}.
 \end{proof}

\begin{lemma} \label{lump}  If $p$ is an idempotent in a Banach algebra 
$A$ then $p  \in {\mathfrak F}_A$
iff $p  \in {\mathfrak r}_A$.    \end{lemma}  \begin{proof}   This is clear from the unital case considered in \cite[Lemma 3.12]{BOZ}.         
\end{proof}

\begin{proposition} \label{wh} If $A$ is a Banach algebra  and $x \in {\mathfrak r}_A$,
then ba$(x)$ has a bai in ${\mathfrak F}_A$.  Hence any 
 weak* limit point of this bai is a mixed identity 
residing in ${\mathfrak F}_{A^{**}}$.   Indeed 
$(x^{\frac{1}{n}})$ is a bai for ${\rm ba}(x)$ in ${\mathfrak r}_A$,
and $({\mathfrak F}(x)^{\frac{1}{n}})$ is a bai for ${\rm ba}(x)$ in ${\mathfrak F}_A$.
  \end{proposition}

\begin{proof}     If $x \in {\mathfrak r}_A^B$ then the proof of \cite[Proposition 3.17]{BOZ} shows that
 ba$(x)$ has a bai in ${\mathfrak F}_A^B$, and hence any 
 weak* limit point of this bai is a mixed identity 
residing in ${\mathfrak F}_{A^{**}}^{B^{**}} \subset {\mathfrak F}_{A^{**}}$.   Indeed 
$(x^{\frac{1}{n}})$ is a bai for ${\rm ba}(x)$ in ${\mathfrak r}_A^B$,
\end{proof}

The following new observation 
is a simple consequence  of the above which we will need later.  

\begin{corollary} \label{eab}  If $A$ is a nonunital Banach algebra and if $E$ and $F$ are subsets of ${\mathfrak r}_A$ then 
$\overline{EA} = \overline{EB}$,  $\overline{AF} = \overline{BF}$, and $\overline{EAF} = \overline{EBF}$,  
 where $B$ is any unitization of $A$.     \end{corollary} 

\begin{proof}   The first follows from the following 
fact: if $x \in {\mathfrak r}_A$  then 
$$x \in \overline{xA} = \overline{{\rm ba}(x) \, A} =  \overline{xB},$$ since by Cohen factorization 
$x \in {\rm ba}(x) = {\rm ba}(x)^2 \subset \overline{xA}$.    The other two are similar.  \end{proof}

We now turn to the support projection of an element, encountered in the Aarnes--Kadison theorem 
\ref{AaKa}.     In an operator algebra or Arens regular Banach algebra things are cleaner (see \cite{BRI,BRII,BBS}).  
For a  Banach algebra $A$ and
$x \in {\mathfrak r}_A$, we    
write $s(x)$ for the weak* Banach limit of $(x^{\frac{1}{n}})$
in $A^{**}$.    That is $s(x)(f) = {\rm LIM}_n \, f(x^{\frac{1}{n}})$ for 
$f \in A^*$, where LIM is a Banach limit.    It is easy to see that $x s(x) = s(x) x = x$, by applying these to
$f \in A^*$.  Hence $s(x)$  is a mixed  identity of ${\rm ba}(x)^{**}$, and 
is idempotent.
By the Hahn--Banach theorem it is easy to see that 
$s(x) \in \overline{{\rm conv}( \{ x^{\frac{1}{n}} : n \in \Ndb \})}^{w*}$.
In $x \in {\mathfrak r}_A^B $ then by an argument after \cite[Proposition 3.17]{BOZ}  we have $s(x) \in 
{\mathfrak F}_{B^{**}} \cap A^{**} = {\mathfrak F}_{A^{**}}^{B^{**}} \subset 
 {\mathfrak F}_{A^{**}}$.
If ${\rm ba}(x)$ is Arens regular then $s(x)$ will be the identity of ${\rm ba}(x)^{**}$.

We call $s(x)$ above a {\em support} idempotent of $x$,
or a (left) support idempotent of $\overline{xA}$
(or a (right) support idempotent of $\overline{Ax}$).   
The reason for this name is the following result.
   
\begin{corollary}  \label{lba}  If  $A$ is a  
Banach algebra, and  $x \in {\mathfrak r}_A$
then $\overline{xA}$ has a left bai in ${\mathfrak F}_A$ and 
$x \in \overline{xA} = s(x) A^{**} \cap A$ and $(xA)^{\perp \perp} = s(x) A^{**}$. 
(These products are with respect to the second Arens product.)
  \end{corollary}

\begin{proof}    The proof of \cite[Corollary 3.18]{BOZ} works, and gives that  $\overline{xA}$ has a left bai in ${\mathfrak F}^B_A$ if $x \in {\mathfrak r}_A^B$.
  \end{proof}

As in \cite[Lemma 2.10]{BRI} and \cite[Corollary 3.19]{BOZ} we have:

\begin{corollary}  \label{supp3}  If $A$ is a Banach
 algebra, and $x, y \in {\mathfrak r}_A$, 
then $\overline{xA} \subset \overline{yA}$
iff $s(y) s(x) = s(x)$.  In this case $\overline{xA} = A$
iff   $s(x)$ is a left identity for $A^{**}$.   (These products are with respect to the second Arens product.)
\end{corollary}

As in \cite[Corollary 2.7]{BRI} we have:

\begin{corollary}  \label{perm}  Suppose that $A$ is a 
subalgebra of a Banach
 algebra $B$.  If  $x \in A \cap {\mathfrak r}_B$, then
the support projection of $x$ computed in $A^{**}$ is
the same, via the canonical embedding $A^{**} \cong A^{\perp \perp}
\subset B^{**}$, as the support projection of $x$ computed in $B^{**}$.
\end{corollary}

In Section 2 we mentioned the paper of Kadison and Pedersen \cite{KP} initiating the development of a  comparison theory for elements in $C^*$-algebras  generalizing the von Neumann  equivalence  of projections.   Again positivity and properties of the positive cone are key to that work.    Admittedly 
their algebras were monotone complete, but many later authors have taken up this theme, with various versions
of equivalence or subequivalence of  elements in general
$C^*$-algebras  (see for example \cite{Bla}
or \cite{ORT,APT} and references therein).
Indeed recently the study of Cuntz equivalence and subequivalence within the context of the Elliott program 
has become one of the most important areas of $C^*$-algebra
theory.        In \cite{BNII} Neal and the author began a program of generalizing
basic parts of the theory of
comparison, equivalence, and subequivalence, to the setting of
general operator algebras.  In that paper we focused on comparison of elements in $\Rdb^+ {\mathfrak F}_A$, but  we proved some lemmas in \cite{BRord} that show that everything should work 
for elements in ${\mathfrak r}_A$.   In particular, we follow the lead of Lin, Ortega, R{\o}rdam, and Thiel \cite{ORT} in studying these equivalences, etc.,   in terms of the roots and 
support projections $s(x)$ discussed in this section above, or in terms of
module isomorphisms of  (topologically) principal modules of the form $\overline{xA}$ studied below.    There is
a lot more work needed to be done here, our paper was simply the first steps.
%QQ
Also, we have not tried to see if any of this
generalizes to larger classes of Banach algebras.   Much of our theory in  \cite{BNII} depends on facts for $n$th roots of real positive elements.  Thus we would expect that a certain
%QQ
portion of this theory generalizes to Banach algebras using the facts about roots summarized in Section \ref{poroo}.

\section{Structure of ideals and HSA's}    \label{AKba}

We recall that  an element $x$ in an algebra $A$ is {\em pseudo-invertible} in $A$ if 
there exists $y \in A$  with $xy x = x$.   The following result (which is the non-approximately unital case
of \cite[Theorem 3.21]{BOZ}) 
 should be compared with the $C^*$-algebraic version of the result due to  Harte and Mbekhta \cite{HM,HM2},
and to the earlier version of the result in the operator algebra case (see particularly \cite[Section 3]{BRI}, and 
\cite[Subsection 2.4]{BRord} and \cite{BRIII}).
 
\begin{theorem}   \label{ws}  Let $A$ be
 a Banach algebra, 
and
 $x \in {\mathfrak r}_A$.  The following
are equivalent:
  \begin{itemize}
\item [(i)] $s(x) \in A$,
\item [(ii)]  $xA$ is closed,
\item [(iii)] $Ax$  is closed,
\item [(iv)]  $x$ is pseudo-invertible in $A$,
 \item [(v)]  $x$ is invertible in ${\rm ba}(x)$.
\end{itemize}
Moreover, these conditions imply
 \begin{itemize}
\item [(vi)] $0$ is isolated in, or absent from, ${\rm Sp}_A(x)$.
\end{itemize}
Finally, if  ${\rm ba}(x)$ is semisimple
then {\rm (i)--(vi)} are equivalent.
\end{theorem}

 \begin{proof}  The first five equivalences are just as in \cite[Theorem 3.21]{BOZ}; as is 
the assertions regarding  (vi), since there we may assume $A$ is unital by definition of spectrum and because of the form of  (v).
\end{proof}

The next results  (which are the non-approximately unital cases 
of results in \cite[Section 3]{BOZ})   follow from Theorem   \ref{ws}  just as the approximately unital cases did in 
 \cite{BOZ}, which in turn often rely on earlier arguments from e.g.\ \cite{BRI}:

\begin{corollary} \label{oka}
If $A$ is a closed 
subalgebra of a unital Banach algebra $B$, 
and if  $x \in{\mathfrak r}_A^B$,  then $x$ is invertible in $B$ 
iff  $1_B \in A$ and $x$
is invertible in $A$, and iff  ${\rm ba}(x)$ contains $1_B$; and in this case $s(x) = 1_B$.
\end{corollary}

\begin{corollary} \label{Aha3}   Let $A$ be a Banach algebra.  A closed right ideal $J$ of $A$
is of the form $x A$ for some $x \in {\mathfrak r}_A$  iff
$J  = qA$ for an idempotent 
 $q \in {\mathfrak F}_A$.
\end{corollary}

\begin{corollary} \label{Aha2}   If a  nonunital 
Banach  algebra
$A$ contains a nonzero $x \in {\mathfrak r}_A$ with $xA$ closed,
 then $A$ contains  a nontrivial 
idempotent in ${\mathfrak F}_A$.  
  If a Banach  algebra $A$ has  no 
left identity, then $x A  \neq A$ for all $x \in {\mathfrak r}_A$.
\end{corollary}

In \cite{BHN} we generalized 
 the concept of {\em hereditary subalgebra} (HSA), an important tool 
in $C^*$-algebra theory, to operator algebras, 
and established that  the basics of the $C^*$-theory of HSA's is still true.     Now of course HSA's need not be selfadjoint, but  are still norm closed   approximately unital {\em inner ideals} in $A$, where by the latter term 
we mean a subalgebra $D$ with $D A D \subset D$.    
Generalizing Theorem \ref{AaKa2} above, we showed in \cite{BRI,BRII} that HSA's and right ideals with left cais in operator algebras
are   {\em  manifestations} of 
our cone ${\mathfrak r}_A$, or if preferred,   ${\mathfrak F}_A$ or the  `nearly positive' elements.
We now discuss some aspects of this in the case of Banach algebras from 
\cite{BOZ}, and mention 
some of what is still true in that setting.  In  particular we study 
the relationship between HSA's and 
  one-sided ideals
with one-sided approximate identities.   
Some aspects of this relationship 
is problematic for general Banach algebras (see \cite[Section 4]{BOZ}), but 
it works much better in 
separable algebras.   
As we said around Theorem \ref{AaKaSin},  our work is closely related to  
the results of Sinclair and others on the Cohen factorization method (see e.g.\ \cite{Sinc,Est}), which does include some similar sounding but different results.

We define a {\em right ${\mathfrak F}$-ideal} (resp.\ left ${\mathfrak F}$-ideal) in a 
Banach algebra $A$  to be a closed right (resp.\ left) ideal with a left (resp.\ right) bai in ${\mathfrak F}_A$
(or equivalently, by Corollary \ref{hasbi}, in ${\mathfrak r}_A$).    
Henceforth in this section, by a {\em hereditary subalgebra} (HSA) of $A$ we will mean 
 an inner ideal  $D$ with a two-sided  bai in ${\mathfrak F}_A$
(or equivalently, by Corollary \ref{hasbi}, 
in ${\mathfrak r}_A$).
Perhaps these should be called ${\mathfrak F}$-HSA's to avoid confusion with the notation of 
\cite{BHN,BRI} where one uses cai's instead of bai's, but for brevity we shall use the 
shorter term.    And indeed for operator algebras (the setting of \cite{BHN,BRI})
the two notions coincide, and also right and left 
${\mathfrak F}$-ideals are just the 
r-ideals and $\ell$-ideals of those papers 
(see Corollary \ref{Aha5}).
Note that a HSA $D$ induces a pair of right and left ${\mathfrak F}$-ideals $J = DA$ and $K = AD$.   
Using the proof in \cite[Lemma 4.2]{BOZ} we have:

\begin{lemma}  \label{zH}  If $A$ is a Banach algebra, and $z \in  {\mathfrak F}_A$,
set  $J = \overline{zA}$, 
 $D = \overline{zAz}$, and $K = \overline{Az}$. \ Then $D$ is a HSA in $A$ 
and $J$ and $K$ are the induced right and left ${\mathfrak F}$-ideals mentioned above.
\end{lemma}  

At this point we have to jettison ${\mathfrak F}_A$ and ${\mathfrak r}_A$ as defined at the start of Section \ref{poroo},
   if $A$ is not approximately unital, because the remaining results are endangered if 
${\mathfrak F}_A$ and ${\mathfrak r}_A$ are not  closed and convex. 
Indeed most 
 results in Sections 
4 and 7 of \cite{BOZ} would seem to  need  ${\mathfrak F}_A$ and ${\mathfrak r}_A$ to be replaced 
by  
${\mathfrak F}^B_A$ and ${\mathfrak r}^B_A$ for a fixed unital Banach algebra $B$ containing $A$ as a subalgebra.  That is, 
we need to fix a particular unitization of $A$, not consider all unitizations 
simultaneously.    Of course if $A$ is
an operator algebra then there is a unique unitization, hence all this is redundant.  
(As we said early in Section \ref{poroo}, 
we could also fix this problem by redefining
${\mathfrak F}_A = \cap_B \, {\mathfrak F}^B_A$ and ${\mathfrak r}_A  = \cap_B \,  {\mathfrak r}^B_A,$
 the intersections taken over all $B$ as above.  
Everything below would then look cleaner but
may be less useful.)
Thus we define a {\em right ${\mathfrak F}^B$-ideal} (resp.\ left ${\mathfrak F}^B$-ideal) in a 
  $A$  to be a closed right (resp.\ left) ideal with a left (resp.\ right) bai in ${\mathfrak F}^B_D$
(or equivalently, by Corollary \ref{hasbi}, in ${\mathfrak r}^B_D$).     Note that one-sided ${\mathfrak F}^B$-ideals in $A$ are exactly subalgebras of $A$ which are  one-sided ${\mathfrak F}^B$-ideals in 
$A+ \Cdb 1_B$ in the sense of \cite[Section 4]{BOZ}. 

We define a  $B$-{\em hereditary subalgebra} (or $B$-HSA for short) of $A$ to 
 an inner ideal  $D$ in $A$ with a two-sided  bai in ${\mathfrak F}^B_D$
(or equivalently, by Corollary \ref{hasbi}, 
in ${\mathfrak r}^B_D$).    Note that $B$-HSA's in $A$ are exactly subalgebras of $A$ which are HSA's in 
$A+ \Cdb 1_B$ in the sense of \cite[Section 4]{BOZ}.

Again a $B$-HSA $D$ induces a pair of right and left ${\mathfrak F}^B$-ideals $J = DA$ and $K = AD$.   
Lemma \ref{zH} becomes: If $z \in  {\mathfrak F}^B_A$,
set  $J = \overline{zA}$, 
 $D = \overline{zAz}$, and $K = \overline{Az}$. \ Then $D$ is a $B$-HSA in $A$ 
and $J$ and $K$ are the induced right and left ${\mathfrak F}^B$-ideals mentioned above.

Because of the facts at the end of the second last paragraph, and because of Corollary \ref{eab}, 
in the following four results we can assume that $A$ is unital, in which case the proofs are in \cite{BOZ}.
These results are all stated for a Banach algebra with unitization $B$, but they could equally well
be  stated for  a closed subalgebra of a unital Banach algebra $B$.

\begin{lemma} \label{oleftc}  Suppose that $J$ is a right ${\mathfrak F}^B$-ideal in a  
Banach algebra with unitization  $B$.   For every compact subset $K \subset J$, there exists 
$z \in J \cap {\mathfrak F}^B_A$ with $K \subset z J \subset zA$.
\end{lemma}  

Applying this lemma gives the  first assertion
in the following result, taking  $K = \{ \frac{1}{n} e_n \} \cup \{ 0 \},$ where $(e_n)$ is the left bai.  Taking $K = \{ \frac{d_n}{n \Vert d_n \Vert} \} \cup \{ 0 \}$ where $\{ d_n \}$ is a countable dense set gives the second.   
  
\begin{corollary} \label{ctrid}  Let $A$ be a Banach algebra with unitization  $B$. 
   The  closed right ideals of $A$ with a countable 
left bai in   ${\mathfrak r}^B_A$ are precisely the (topologically) `principal right ideals'
$\overline{zA}$ for some $z \in {\mathfrak F}^B_A$ which is also in the ideal.   Every separable right ${\mathfrak F}^B$-ideal 
is of this form.
\end{corollary}

\begin{corollary}  \label{pred}    Let $A$ be a Banach algebra with unitization  $B$. 
 The right ${\mathfrak F}^B$-ideals  in $A$ are precisely the 
closures of  increasing unions of right ideals in $A$ of the form $\overline{zA}$ for some $z \in {\mathfrak F}^B_A$. \end{corollary}

We say that a right module $Z$ over $A$ is
{\em algebraically countably generated} (resp.\ {\em algebraically finitely
generated})  
over $A$ if there exists a
countable (resp.\  finite set) $\{ x_k \}$ in $Z$ such that
every $z \in Z$ may be written as a finite sum $\sum_{k=1}^n \, x_k a_k$
for some $a_k \in A$.   Of course any algebraically finitely
generated is algebraically countably generated.

\begin{corollary} \label{Aha4}    Let $A$ be a Banach algebra with unitization  $B$.    A right ${\mathfrak F}^B$-ideal $J$ in $A$
is algebraically countably generated as a right module over $A$   iff   $J  = q A$ for an idempotent $q \in {\mathfrak F}^B_A$.  This
is also equivalent to $J$ being algebraically countably generated as a right module over $A + \Cdb 1_B$.
\end{corollary}  

The following was not stated in \cite{BOZ}.

\begin{corollary} \label{Aha5}  If $A$ is an operator algebra,
a closed subalgebra of a unital operator algebra $B$,
 then right and left ${\mathfrak F}^B$-ideals in $A$ are just the 
r-ideals and $\ell$-ideals in $A$ of {\rm \cite{BHN,BRI}},
and $B$-HSA's in $A$ are just the HSA's in $A$ of those references.
 \end{corollary}  

\begin{proof}  By Corollary \ref{pred} a right ${\mathfrak F}^B$-ideal
is the closure of  an increasing union of right ideals in $A$ of the form $\overline{zA}$ for $z \in {\mathfrak F}_A$.  However this is the 
characterization of r-ideals from \cite{BRI}.  Similarly for
the left ideal case.  A similar argument works 
for the HSA case using Corollary \ref{pred2}; alternatively, if $D$ is a $B$-HSA then 
$D$ has a bai from ${\mathfrak F}_A$.   By Corollary \ref{hasbi4},
 $D$ has a cai.   \end{proof}  

If $A$ is a Banach algebra with unitization $B$ it would be nice to say that the  
right ${\mathfrak F}^B$-ideals  in $A$ are precisely the sets of form 
$\overline{EA}$ for a subset $E \subset  {\mathfrak F}^B_A$ (or equivalently, $E \subset  {\mathfrak r}^B_A$).
One direction of this is obvious: just take $E$ to be the bai in ${\mathfrak F}^B_A$ (resp.\ ${\mathfrak r}^B_A$).
However the other direction is false in general Banach algebras, although it does hold in operator algebras \cite{BRI}
and commutative Banach algebras.   (Another characterization of closed ideals with bai's in commutative Banach algebras may be found in
\cite{LU}.)

That   $\overline{EA}$ is a right ${\mathfrak F}^B$-ideal in $A$ if $A$ is a commutative Banach algebra
and  $E \subset  {\mathfrak F}^B_A$, follows from Theorem 7.1 in \cite{BOZ} after noting that by 
Corollary \ref{eab} we may replace $A$ by $A + \Cdb 1_B$.    The key part of the proof of  Theorem 7.1 in \cite{BOZ}
 is to 
show that for any finite subset $G$ of $E$ there exists  an element $z_G \in {\mathfrak F}^B_A \cap  \overline{EA}$ 
with $\overline{G A} = \overline{z_G A}$.    Indeed one can take $z_G$ to be the average of the elements in $G$.  Then the net $(z_G^{\frac{1}{n}})$, indexed by the finite subsets $G$ of $E$
and $n \in \Ndb$, is easily seen to be a left bai in  $\overline{EA}$ from ${\mathfrak F}^B_A$.  An application of this: for such subsets $E$ of an operator algebra
or commutative Banach algebra $A$, the Banach algebra 
generated by $E$ has a bai in ${\mathfrak F}^B_A$.     This follows from the argument above since the $z_G$ above 
are in the convex hull of $E$, hence the bai $(z_G^{\frac{1}{n}})$ is in the Banach algebra 
generated by $E$.     In particular, if $A$ is generated as a Banach algebra by ${\mathfrak r}^B_A$, then $A$ has 
a bai, and this bai may be taken from ${\mathfrak r}^B_A$.     (The present paragraph is a summary of the results in 
\cite[Section 7]{BOZ}, and a generalization of these results to the case that $A$ is not approximately unital.)

Unless explicitly said to the contrary,
all the remaining results in this section are again 
generalizations to general Banach algebras of results from \cite{BOZ}.  Some of these
were proved earlier in the operator algebra case 
by  the author and Read.
Again for their proofs we can assume  that $A$ is unital,
and appeal to the matching results in \cite{BOZ}.

\begin{lemma}  \label{preot}   Let $A$ be a Banach algebra with unitization  $B$. Let $D$ be a 
closed subalgebra of $A$.   If $D$ has a bai from ${\mathfrak F}^B_A$, then for every compact subset $K\subset D$,
 there is an $x \in  {\mathfrak F}_D^B$ such that $K\subset xDx \subset xAx$. \end{lemma}  

As in the proof sketched for Corollary \ref{ctrid} this leads to: 

 \begin{theorem} \label{otter} Let $A$ be a Banach algebra with unitization  $B$, and 
 let $D$ be an inner ideal in $A$.
Then  $D$  has a
 countable bai from ${\mathfrak F}^B_D$ 
(or equivalently, from ${\mathfrak r}_D^B$)
 iff there exists an element 
$z \in  {\mathfrak F}_D^B$ with $D =  \overline{z A z}$.  Thus  $D$ is of the form in Lemma {\rm \ref{zH}}, and   
such $D$ has a countable commuting 
bai from ${\mathfrak F}^B_D$, namely $(z^{\frac{1}{n}})$.  Any separable inner ideal in $A$ with a bai from ${\mathfrak r}^B_D$ is of this form.  
 \end{theorem}  

From this most of the following
 generalization of the Aarnes--Kadison theorem (see Theorem \ref{AaKa}) is immediate.
By a {\em strictly real  positive element} in (v) below, we mean an element $x \in A$ such that Re $\varphi(x)  > 0$ for
all states $\varphi$ of $A$ which do not vanish on $A$.
In \cite{BRI, BRord} we generalized some basic aspects of strictly positive elements in $C^*$-algebras to operator 
algebras.     
The following is mostly in \cite{BOZ,BRord}, and relies on ideas from \cite{BRI}.  

\begin{corollary}[Aarnes--Kadison type theorem] \label{ottertoo}   If $A$ is a Banach algebra 
then the following are equivalent: 
\begin{itemize} 
\item [(i)]   
  There  exists
an $x \in {\mathfrak r}_A$ with $A = \overline{x A x}$.
\item [(ii)]   There  exists
an $x \in {\mathfrak r}_A$ with $A = \overline{xA} = \overline{Ax}$.
\item [(iii)]    There  exists
an $x \in {\mathfrak r}_A$ with $s(x)$ a mixed identity for $A^{**}$.
\end{itemize}  
 If $B$ is a unitization of $A$ then items {\rm (i), (ii),} or {\rm (iii)} above hold 
with $x \in {\mathfrak r}^B_A$ iff 
\begin{itemize}
\item [(iv)]  $A$ has a sequential    bai from ${\mathfrak r}^B_A$.
\end{itemize}  
The approximate identity in {\rm (iv)} may be taken to be commuting, indeed it may be taken to be  
$(x^{\frac{1}{n}})$ for the last mentioned element $x$.  If $A$ is  separable and has a bai in ${\mathfrak r}^B_A$  then $A$ satisfies {\rm (iv)} and hence all of the above.
 Moreover if $A$ is an operator algebra then {\rm (i)}--{\rm (iv)} are each equivalent to:   \begin{itemize}
 \item [(v)]  $A$  has a strictly  real positive element,   \end{itemize}   
and any of these imply that the operator algebra $A$ has a sequential 
real positive cai.   
\end{corollary}

Again, ${\mathfrak r}$
can be replaced by ${\mathfrak F}$ throughout this result, or in any of the items (i) to (v).   

The proof of Corollary \ref{ottertoo} is mostly in \cite{BOZ,BRord}, and 
relies partly 
on ideas from \cite{BRI}.  In the operator algebra case, if (ii) holds with $x \in {\mathfrak F}_A$ then $((\frac{1}{2} x)^{\frac{1}{n}})$ is a cai for $A$ in
$\frac{1}{2}  {\mathfrak F}_A$ by \cite[Section 3]{BRI}, and $s(x) = 1_{A^{**}}$.  So $x$  is a  strictly  real positive element by \cite[Lemma 2.10]{BRI}.
Conversely, if an operator algebra $A$ has  a strictly  real positive element then it is explained 
in the long discussion before \cite[Lemma 3.2]{BRord} how to adapt the proof of \cite[Lemma 2.10]{BRI} to show that 
 (iv) holds, hence (ii), and hence $A$ has a sequential real positive
cai by e.g.\
\cite[Corollary 3.5]{BRII}, or by our earlier
Corollary \ref{Aha5}.

\begin{corollary}  \label{pred2}  The $B$-HSA's in a  Banach algebra $A$ with unitization $B$ are exactly the 
closures of  increasing unions of HSA's of the form $\overline{zAz}$ for  $z \in {\mathfrak F}^B_A$. \end{corollary}

{\em Acknowledgements.}   We thank the referee for his comments.   

\bibliographystyle{amsalpha}

\end{document}